\documentclass{article}

\usepackage{amssymb,amsfonts,amsmath,latexsym,xcolor,amscd} 
\usepackage{amsthm}
\usepackage{empheq}
\usepackage[all,cmtip]{xy}
\usepackage{url}
\usepackage{graphicx} 
\usepackage{cases}
\usepackage{appendix}
\usepackage{tikz}
\usetikzlibrary{arrows.meta,positioning,calc}

\newtheorem{theorem}{Theorem}[section]
\newtheorem{lemma}[theorem]{Lemma}
\newtheorem{corollary}[theorem]{Corollary}
\newtheorem{proposition}[theorem]{Proposition}

\newtheorem{definition}[theorem]{Definition}

\newtheorem{remark}[theorem]{Remark}

\newtheorem*{maintheorem}{Main theorem}

\numberwithin{equation}{section}

\def\a{\alpha}
\def\e{\varepsilon}
\def\R{\Re\mathfrak{e} \,}
  
\def\RR{\mathbb{R}}
\def\CC{\mathbb{C}}

\def\la{\lambda}

\def\Ga{\Gamma}
\def\t{\theta}

\def\uon#1{\nu_{#1}}

\def\dom{\Omega}

\def\pOm{\partial\Omega}
\def\Om{\Omega}

\def\d{d}
\def\dx{\d x}
\def\dy{\d y}

\newcommand{\diam}{\operatorname{diam}}
\newcommand{\divg}{\operatorname{div}}

\begin{document}

\title{Identification of generic polygonal domains by integral-geometric invariants}

\author{Jun O'Hara\footnote{Supported by JSPS KAKENHI Grant Number 23K03083.}}
\maketitle

\begin{abstract}
We study the identification of planar polygonal domains from the distribution 
of distances between pairs of points. 
We prove that this distribution determines a generic polygonal domain, not necessarily convex, up to Euclidean isometry within the class of generic polygonal domains. 
The same result holds for the direction-averaged covariogram and the Riesz energy function. 
Since the direction-averaged covariogram carries strictly less information than the covariogram itself, this does not follow from reconstruction results based on the covariogram. 
The proof uses an equivalent boundary distance distribution weighted by the scalar product of the outer unit normals. 
We give explicit relations between these quantities. 
For polygons, the singularities of the weighted boundary distribution encode intervertex distances, interior heights, and local angle data. 
Under the genericity assumptions, these data recover the side lengths, their cyclic order, and the signed exterior angles. 
In the convex case, the chord length distribution is equivalent as well. 
This extends Waksman's results on reconstruction of convex polygons to the non-convex case, under different genericity assumptions. 
\end{abstract}



\setcounter{tocdepth}{3}
%\tableofcontents

%!!!!!!!!!!!!!!!!!!!!!!!!!!!!!!!!!!!!!%%%%%%%%%%%%%%%%%%%%%%%%%%%%%%%%%%%%
\section{Introduction}
%!!!!!!!!!!!!!!!!!!!!!!!!!!!!!!!!!!!!!%%%%%%%%%%%%%%%%%%%%%%%%%%%%%%%%%%%%

We study the extent to which a planar domain can be recovered from the distances between pairs of its points. 
Throughout this paper, a \emph{polygonal domain} is the closure of the bounded region enclosed by a simple closed polygonal curve, with no two consecutive sides collinear. 
Convexity is not assumed. 
For such a domain $\Omega$, its \emph{interpoint distance distribution}
(IDD) is 
\[
  \Psi_\Omega(r)
  =\int_\Omega\int_\Omega
       \mathbf 1_{\{|x-y|\le r\}}\,\dx\,\dy,
  \qquad r\ge0.
\]
We use the unnormalized distribution, so that
$\Psi_\Omega(r)=\operatorname{Vol}(\Omega)^2$ for
$r\ge\operatorname{diam}\Omega$.

\begin{maintheorem}[Theorem \ref{thm_generic}] 
Let $\Omega_1$ and $\Omega_2$ be generic planar polygonal domains. 
If $\Psi_{\Omega_1}=\Psi_{\Omega_2}$, then $\Omega_1$ and $\Omega_2$ coincide up to isometry. 
\end{maintheorem}

The genericity conditions are stated precisely in Definition \ref{def_generic}. 
Their role is to ensure that the geometric information carried by the distribution can be separated. 
The \emph{critical lengths} are the intervertex distances and the distances from vertices to sides whose perpendicular feet lie in the relative interiors of the sides. 
We require that these lengths are dissociated: no nonzero relation with coefficients in $\{0,\pm1\}$ holds among them. 
We also exclude coincidences between critical lengths and exterior heights at adjacent vertices, and require
the values $|\tan\theta_i|$ to be distinct, where $\theta_i$ are the signed exterior angles and $|\tan(\pm\pi/2)|=\infty$. 
The angles $\theta_i$ themselves are signed, and their signs are part of what is reconstructed. 
The genericity condition involves only $|\tan\theta_i|$ because the jump formulae \eqref{jump_nu} and \eqref{jump_nu_3} are invariant under $(\alpha,\beta)\mapsto(-\alpha,-\beta)$ (see \eqref{formula_J_algebraic}), so the data determine the tangents only up to a simultaneous change of sign, which corresponds to a reflection. 
The theorem asserts uniqueness within this generic class; no assertion is made here about a nongeneric domain having the same data as a generic one.

\paragraph{Related reconstruction problems.} 
(i) For a convex planar body, where a {\em body} in Euclidean space refers to a compact subset equal to the closure of its interior, the IDD is equivalent to the chord length distribution
$\Phi_\Omega$; see Mat\'ern \cite{M} and Santal\'o \cite{San2}. 
Blaschke asked whether a convex body is determined by its chord-length moments. 
Equivalently, one may prescribe the full chord length distribution, since a
finite measure with compact support is determined by its nonnegative integer
moments. 
Mallows and Clark \cite{MC} constructed noncongruent convex polygons with the same chord length distribution. 
Gates \cite{G} established identification results for triangles and convex quadrilaterals, and Waksman \cite{W} proved reconstruction for a generic class of convex polygons.
Our result concerns nonconvex polygons as well, with genericity assumptions
different from Waksman's.

\smallskip
(ii) Another related inverse problem concerns the covariogram
\begin{equation}\label{def_covariogram}
g_\Om(x) = \operatorname{Vol}(\Om \cap (\Om+x)) \qquad (x\in\RR^2).
\end{equation}

Matheron's conjecture \cite{M2} asks whether the covariogram determines a convex
body within the class of convex bodies, up to translations and reflections in a point. 
In the plane, Nagel \cite{N} proved this for convex polygons (see also \cite{Bia02} for a shorter proof), and
Averkov and Bianchi \cite{AB} established the result for all convex bodies. 
Schmitt \cite{Sch} gave a constructive reconstruction procedure for a generic class of polygons, not necessarily convex, using directional derivative information from the covariogram. 
General nonconvex polygons need not be determined by their covariograms \cite{BBD,GGZ}.

The connection with the present problem is
\[
  \Psi_\Omega(r)=\int_{|z|\le r}g_\Omega(z)\,\d z
               =2\pi\int_0^r t\,\bar g_\Omega(t)\,\d t,
  \qquad
  \bar g_\Omega(r)=\frac1{2\pi}\int_0^{2\pi}
      g_\Omega(r\cos\theta,r\sin\theta)\,\d\theta. 
\]
Thus the IDD is equivalent to the direction-averaged covariogram, whereas the
full covariogram retains the directional variable. 
Our reconstruction problem asks what can still be recovered after this variable has been integrated out.
Related explicit calculations of isotropic correlation functions for polygons and polyhedra appear in \cite{C09,C10, C20}. %; our focus is the extraction of geometric data sufficient for reconstruction. 

\smallskip
(iii) The chord Minkowski problems introduced by Lutwak, Xi, Yang and Zhang \cite{LXYZ} provide another related perspective. 
Their chord measures arise as first variations of chord-length moment integrals. 
These problems prescribe a measure on the sphere of normal directions, rather than the direction-averaged length distribution considered here. 
For broader background on integral and
stochastic geometry, see \cite{Matheron,Mo,San2,SW}; for covariograms and their relation to aperiodic order, see \cite{BG,Bianchi}.

\paragraph{Boundary formulation.} 
Let $\Gamma=\partial\Omega$ and let $\nu_x$ denote its outer unit normal,
defined for almost every $x\in\Gamma$. 
We use the normal-weighted boundary interpoint distance distribution
\[
  \Psi_{\nu;\Gamma}(r)
  =\int_\Gamma\int_\Gamma
       \mathbf 1_{\{|x-y|\le r\}}
       \langle\nu_x,\nu_y\rangle\,\dx\,\dy.
\]
The direct relations established in Section 2 are
\[
  \Psi_{\nu;\Gamma}=-2\pi r\bar g_\Omega'(r)
       =\frac{\Psi_\Omega'(r)}r-\Psi_\Omega''(r),
  \qquad r>0,
\]
and
\[
  \Psi_\Omega(r)
     =\frac12\int_0^r t\,\Psi_{\nu;\Gamma}(t)\,dt
      +\frac{r^2}{2}\int_r^\infty\frac{\Psi_{\nu;\Gamma}(t)}t\,\d t.
\]
These formulae hold without convexity and show that the body IDD and the normal-weighted boundary IDD determine each other. 
Their proof uses the
distributional identity $\Delta g_\Omega=-\eta$, where $\eta$ is the push-forward
of $\langle\nu_x,\nu_y\rangle\,\dx\,\dy$ under $(x,y)\mapsto x-y$. 
This is related to the covariogram derivative formulae in \cite{BBD,Sch}. 

\paragraph{Outline of the proof.} 
For a polygon, $\Psi_{\nu;\Gamma}$ is a sum of contributions from pairs of sides. 
By parametrizing the two sides we can identify their product with a rectangle. 
Then a distance sublevel set in this rectangle is an elliptical region, possibly degenerate. 
Changes occur when the level set passes through a corner or becomes tangent to a side. 
This gives the following reconstruction scheme. 

\smallskip
(i) {Critical lengths and side lengths.} 
Under the noncoincidence condition, the singularities of $\Psi_{\nu;\Gamma}''$ identify the critical lengths. 
The identity $\operatorname{Length}(\Gamma)=\Psi_{\nu;\Gamma}'(0^+)/2$ gives the perimeter. 
To ensure that a subset of critical lengths whose sum equals the perimeter is uniquely determined, we adopt the assumption that the set of critical lengths is dissociated. 
At this stage the set of side lengths is known but their cyclic order is unknown. 

\smallskip
(ii) {Local angles and incidence.} 
At each side length, jumps or leading singular coefficients of the derivatives of $\Psi_{\nu;\Gamma}$ give the data on exterior angles at its two endpoints. 
We assume that the values of $|\tan\theta_i|$ are all distinct. This assumption lets us determine which sides share a vertex, and hence the cyclic order of the sides. 
Relative signs are fixed consistently, up to an overall reversal that corresponds to reflection. 

\smallskip
(iii) {Resolving the remaining ambiguity.} 
Once adjacent side lengths are known, the occurrence or absence of the corresponding interior height distinguishes the two possible signed exterior angles with the same tangent. 
We impose the exterior-height condition to exclude accidental coincidences. 
The resulting cyclic sequence of side lengths and signed exterior angles determines the polygon up to Euclidean isometry. 

%\smallskip
\paragraph{The Riesz energy function.} 
The Riesz energy function, defined as the integral of the Riesz potential over the domain viewed as a function of the power, is equivalent to the IDD via the Mellin-Stieltjes transform or through the compact-support moment argument: 
\[
  B_\Omega(s)=\int_\Om\int_\Om {|x-y|}^s\,\dx\,\dy
=\int_0^\infty r^s\,d\Psi_\Omega(r),\qquad \R s>-2. 
\]
Furthermore, by Stokes' theorem, the Riesz energy function admits a boundary integral representation with normal weighting similar to that of $\Psi_{\nu;\Ga}$. 
Thus the Riesz energy function also provides an alternative route for proving the equivalence between the $\Psi_\Om$ and $\Psi_{\nu;\Ga}$. 

\smallskip
{\bf Organization of the paper}. 
Section 2 collects these relations. Section 3 establishes the local formulae and summarizes the reconstruction data in Corollary \ref{perimeter} and Propositions \ref{from_Psi_to_polygon} and \ref{Psi<->angles}. Section 4 states the precise genericity conditions and completes the reconstruction.

\medskip
{\bf Acknowledgement}: 
The author is deeply grateful to the anonymous referee for helpful suggestions, especially for pointing out the connection to the covariogram and providing relevant references, as well as for suggesting a new approach regarding the equivalence of the invariants. 

After drafting the manuscript, the author used 
OpenAI ChatGPT Pro and Anthropic Claude Opus 5.5/Fable 5 for language editing and internal consistency checks. 
The author reviewed and edited the content as necessary and takes full responsibility for the final version of the manuscript.

\medskip
{\bf Notation}: By $\Omega$ we mean a compact body in $\RR^2$, from Section 3 on, $\Om$ is a polygonal domain, and $\Ga$ denotes its boundary. 
Throughout the paper, $\Omega$ is connected, $\Ga=\pOm$ is simple closed, and the interior angle at each vertex is neither $0$ nor $\pi$. 
By $[PQ]$ we mean a line segment with endpoints $P$ and $Q$. 
By $\operatorname{Vol}$ we mean volume with respect to the natural measure on the space under consideration: area measure on $\Omega$, arclength measure on $\Gamma$, and the corresponding product measure on Cartesian products. 
In integrals, $\d x$ and $\d y$ denote area or arclength measure according to the domain of integration, and $\d x\,\d y$ denotes their product. All these measures are unnormalized. 

%!!!!!!!!!!!!!!!!!!!!!!!!!!!!!!!!!!!!!%%%%%%%%%%%%%%%%%%%%%%%%%%%%%%%%%%%%
\section{Integral-geometric invariants and their equivalence}\label{sec_2} 
%!!!!!!!!!!!!!!!!!!!!!!!!!!!!!!!!!!!!!%%%%%%%%%%%%%%%%%%%%%%%%%%%%%%%%%%%%

%!!!!!!!!!!!!!!!!!!!!!!!!!!!!!!!!!!!!!%%%%%%%%%%%%%%%%%%%%%%%%%%%%%%%%%%%%
\subsection{Definitions}
%!!!!!!!!!!!!!!!!!!!!!!!!!!!!!!!!!!!!!%%%%%%%%%%%%%%%%%%%%%%%%%%%%%%%%%%%%
%
We begin by introducing several classical quantities in integral geometry, along with a variant of one of these. 
Let $X$ denote either a compact body $\Omega$ in $\RR^2$ or its boundary $\Ga=\pOm$. 
Let $r\ge0$ and $Y,W\subset X$. Put 
\[
\begin{array}{rcl}
\displaystyle \Psi_{Y,W}(r)&=&\displaystyle 
 \int_Y\int_W\mathbf1_{\{|x-y|\le r\}}\,\d x\,\d y, \\[4mm]
\displaystyle \Psi_Y(r)&=&\Psi_{Y,Y}(r), 
\end{array}
\]
and call $\Psi_Y(r)$ the {\em interpoint distance distribution} (IDD for short) of $Y$. It is a continuous function of $r$. 

Let $\nu_x$ ($x\in\Ga$) be an outer unit normal vector to $\Ga=\pOm$ at point $x$. 
For $Y,W\subset\Ga$ define the {\em $\nu$-weighted interpoint distance distribution} ($\nu$-weighted IDD for short) by 
\begin{equation}
\begin{array}{rcl}
\Psi_{\nu;Y,W}(r)&=&\displaystyle 
 \int_Y\int_W\mathbf1_{\{|x-y|\le r\}}\langle\nu_x,\nu_y\rangle\,\d x\,\d y, \\[4mm]
\Psi_{\nu;\,Y}(r)&=&\displaystyle \Psi_{\nu;Y,Y}(r).
\end{array}
\end{equation}
By analyzing this $\Psi_{\nu;\,\Ga}$, we reconstruct the polygon. 

\smallskip
Assume $\Omega$ is convex. 
Let $\mathcal{E}_1$ be the set of lines in $\RR^2$, and $\mu$ be the measure on $\mathcal{E}_1$ defined by $d\mu=\d p\wedge \d\theta$, where $p,\theta$ are the polar coordinates of the foot of the perpendicular to the line from the origin. Then $\mu$ is invariant under motions of $\RR^2$ (\cite{San2} page 27). 
Let $\sigma_\Om(\ell)$ be the length of $\Om\cap\ell$, and put 
\[
\begin{array}{rcl}
\mathcal{E}_1(\Omega)&=&\displaystyle \{\ell\in\mathcal{E}_1\,:\,\ell\cap \mathrm{int}\, \Omega\ne\emptyset\}, \\[2mm]
\Phi_\Om(r)&=&\displaystyle \mu\{\ell\in\mathcal{E}_1(\Omega)\,:\,\sigma_\Om(\ell)\le r\}, 
\end{array}
\]
and call $\Phi_\Om$ the {\em chord length distribution} of a convex body $\Om$. 
The function $\ell\mapsto\sigma_\Omega(\ell)$ is the Radon transform of the indicator function of $\Omega$; see Gel'fand, Graev and Vilenkin \cite[Chapter~I, Section~1.1]{GGV}. 

\smallskip

We return to the setting where $\Omega$ is not necessarily convex. 
The {\em direction averaged covariogram} is given by 
\[
\bar g_\Om(r)
=\frac1{2\pi}\int_0^{2\pi}
g_\Om(r\cos\theta,r\sin\theta)\,d\theta,
\]
where $g_\Om$ is covariogram given by \eqref{def_covariogram}.
It is called the {\em isotropic correlation function} in \cite{C09, C10, C20}. 
We note that $\bar g_\Om(0)=\operatorname{Vol}(\Om)$.

\smallskip 
Next we introduce the Riesz energy function. 
Put
\[
B_X(z)=\iint_{X\times X}|x-y|^z\,\d x\d y \qquad (z\in\CC), 
\]
where $\d x$ and $\d y$ denote the $(\dim X)$-dimensional Hausdorff measure of $X$, namely, area measure when $X=\Omega$ and arclength measure when $X=\Gamma$. 
It is well-defined if $\R z>-\dim X$, and the map is holomorphic on this region. 

For the smooth and polygonal sets considered below, $B_X$ admits a meromorphic continuation to $\mathbb C$, with at most simple poles at certain negative integers\footnote{How far $B_X$ can be analytically continued depends on the regularity of the boundary $\Gamma = \partial\Omega$ (\cite{Oball} Corollary 2.4).}. 
The reader is referred to Gel'fand and Shilov \cite[Chapter~I, Sections~1 and 3]{GS} for the meromorphic regularization. 
We denote it by the same symbol $B_X(z)$ and call it the {\em Riesz energy function} (\cite{OS2}), 
since for a body it is the integral of the Riesz potential. 
It was introduced by Brylinski for knots under the name of the {\em beta function} \cite{B}, and studied for closed submanifolds by Fuller and Vemuri \cite{FV}; we use the former terminology throughout 
since our object is the two-dimensional domain $\Om$. 

We define the {\em $\nu$-weighted Riesz energy function} of $\Ga$ (\cite{OS2}) by 
\[
B_{\nu;\,\Ga}(z)=\iint_{\Ga\times \Ga}{|x-y|}^{z}\langle\uon{x},\uon{y}\rangle\,\d x \d y.
\]

%
%!!!!!!!!!!!!!!!!!!!!!!!!!!!!!!!!!!!!!%%%%%%%%%%%%%%%%%%%%%%%%%%%%%%%%%%%%
\subsection{Equivalence of the invariants}\label{subsec_equivalence} 
%!!!!!!!!!!!!!!!!!!!!!!!!!!!!!!!!!!!!!%%%%%%%%%%%%%%%%%%%%%%%%%%%%%%%%%%%%

We say that two invariants are {\em equivalent} if, within the class of domains under consideration, each of them determines the other. 

In this subsection we assume that a domain $\Om$ satisfies the following condition, temporarily relaxing the assumption that $\Gamma$ is a single Jordan curve\footnote{The divergence theorem used in Lemma \ref{lem:key} is valid more generally for sets of finite perimeter (Federer \cite[4.5.6]{F}), so the condition $(*)$ can be relaxed accordingly, however we do not pursue this here.}

\medskip
($\ast$) $\Ga=\partial\Om$ is a finite union of pairwise disjoint Jordan curves, each of which is either $C^\infty$ or polygonal. 

\begin{lemma}\label{lem:key}
Assume $\Omega$ satisfies the condition $(\ast)$. 
Let $\varphi$ be a function of class $C^2$ defined on a neighbourhood of $\Om-\Om=\{x-y\,:\,x,y\in\Om\}$. 
Then 
\begin{equation}
\iint_{\Om\times\Om}(\Delta\varphi)(x-y)\,dx\,dy
=-\iint_{\Ga\times\Ga}\varphi(x-y)\,\langle\nu_x,\nu_y\rangle\,dx\,dy .
\label{eq:star}
\end{equation}
\end{lemma}

\begin{proof}
First fix $y$ and apply the divergence theorem to obtain
\[
\int_\Om(\Delta\varphi)(x-y)\,dx
=\int_\Om\divg_x\bigl[(\nabla\varphi)(x-y)\bigr]dx
=\int_\Ga\bigl\langle(\nabla\varphi)(x-y),\nu_x\bigr\rangle\,dx .
\]
Now fix $x$. Since 
\[
-\divg_y\bigl[\varphi(x-y)\,\nu_x\bigr]=-\bigl\langle\nabla_y\bigl[\varphi(x-y)\bigr],\nu_x\bigr\rangle
=\bigl\langle(\nabla\varphi)(x-y),\nu_x\bigr\rangle,
\]
we have 
\[\int_\Om\bigl\langle(\nabla\varphi)(x-y),\nu_x\bigr\rangle dy
=-\int_\Ga\varphi(x-y)\langle\nu_x,\nu_y\rangle\,dy.\]
Combining the two by Fubini's theorem gives \eqref{eq:star}.
\end{proof}

\begin{proposition}\label{prop:equiv_via_bar_g}
Assume $\Omega$ satisfies the condition $(\ast)$. 
Let $\Psi_\Om, \Psi_{\nu;\Ga}, g_\Om, \bar g_\Om$ be as before. 
Then, for every $r>0$,
\begin{eqnarray}%\label{}
\Psi_\Om'(r)&=&2\pi r\,\bar g_\Om(r), \label{Psi_Om'-bar_g}\\
\Psi_{\nu;\Ga}(r)&=&\displaystyle -2\pi r\,\bar g_\Om'(r)=\frac{\Psi_\Om'(r)}r-\Psi_\Om''(r), \label{Psi_nu-bar_g'}
\end{eqnarray}
and
\begin{equation}
\Psi_\Om(r)=\frac12\int_0^r t\,\Psi_{\nu;\Ga}(t)\,\d t
+\frac{r^2}2\int_r^\infty\frac{\Psi_{\nu;\Ga}(t)}{t}\,\d t .
\label{eq:inverse}
\end{equation}
In particular, $\Psi_\Om$ and $\Psi_{\nu;\Ga}$ determine each other.
\end{proposition}

\begin{proof}
Let $\eta$ be the push-forward of the signed measure $\langle\nu_x,\nu_y\rangle\,\d x\,\d y$ on $\Ga\times\Ga$ under the difference map $(x,y)\mapsto x-y$, i.e. $\eta$ is defined to satisfy 
\begin{equation}\label{eq:push-forward}
\int_{\mathbb R^2}\varphi(z)\,\d\eta(z)
=
\iint_{\Gamma\times\Gamma}
\varphi(x-y)\langle\nu_x,\nu_y\rangle\,\d x\,\d y
\end{equation}
for a test function $\varphi\in C_c^\infty(\mathbb R^2)$. 
Then, $\Psi_{\nu;\Ga}(r)=\eta(\{z:|z|\le r\})$, and as distributions, 
\begin{equation}
\Delta g_\Om=-\eta ,
\label{eq:laplace}
\end{equation}
where $g_\Om$ is the covariogram\footnote{See Remark \ref{rem:2}.}. 
In fact, since $g_\Om$ is the convolution $g_\Om=\mathbf 1_\Om\star\mathbf 1_{-\Om}$, \eqref{eq:star} and \eqref{eq:push-forward} imply 
\begin{equation}\label{eq:Laplacian_g}
\begin{aligned}
\langle\Delta g_\Om,\varphi\rangle
&=\int_{\mathbb R^2}g_\Om(z)\Delta\varphi(z)\,dz\\
&=\iint_{\Omega\times\Omega}\Delta\varphi(x-y)\,\d x\,\d y\\
&=-\iint_{\Gamma\times\Gamma} \varphi(x-y)\langle\nu_x,\nu_y\rangle\,\d x\,\d y \\
&=-\int_{\mathbb R^2}\varphi(z)\,\d\eta(z).
\end{aligned}
\end{equation}

We first note that $\Psi_{\nu;\Ga}$ is continuous. It follows from 
\[
 (\mathcal H^1|_\Gamma\otimes\mathcal H^1|_\Gamma)
       \{(x,y):|x-y|=r\}=0,
\]
where $\mathcal H^1|_\Ga$ is the arclength measure. 
For fixed $r>0$, circles of radius $r$ with distinct centers intersect in at most two points. 
Since $\Gamma$ has finite length, only countably many such circles can have an intersection with $\Gamma$ of positive arclength. 
The displayed equality therefore follows from Fubini's theorem. 

For $\psi\in C_c^\infty((0,\infty))$, use the smooth radial test function
$\varphi(z)=\psi(|z|)$. 
Since 
\[\Delta \varphi = \psi''(r) + \frac{1}{r} \psi'(r)=\frac1r\bigl(r\psi'(r)\bigr)',\]
\eqref{eq:Laplacian_g} implies  
\[
 2\pi\int_0^\infty\bar g_\Om(r)\bigl(r\psi'(r)\bigr)'\,\d r
=2\pi\left\langle(r\bar g_\Om')',\psi\right\rangle
%2\pi\int_0^\infty\psi(r) \bigl(r\bar g_\Om'(r)\bigr)'\,\d r
=-\int_0^\infty\psi(r)\,d\Psi_{\nu;\Ga}(r).
\]
Consequently,
\[
 \bigl(\Psi_{\nu;\Ga}(r)+2\pi r\bar g_\Om'(r)\bigr)'=0
\]
as distributions on $(0,\infty)$.
For $r>\operatorname{diam}\Omega$, one has $\bar g_\Om(r)=0$ and
\[
 \Psi_{\nu;\Ga}(r)=\left|\int_\Gamma\nu_x\,\dx\right|^2=0.
\]
Therefore the integration constant is zero, and
\begin{equation}\label{eq:Psi_nu-bar_g}
\Psi_{\nu;\Ga}(r)
=-2\pi r\,\bar g_\Om'(r)
\end{equation}
as distributions. 
Since $\Psi_{\nu;\Ga}$ is continuous and $\bar g_\Om$ is continuous, it follows that $\bar g_\Om\in C^1((0,\infty))$ and this equality holds pointwise, which implies the first equality of \eqref{Psi_nu-bar_g'}. 

By the change of variables $z=x-y$,
\begin{equation}\label{eq:Psi-bar_g}
 \Psi_\Om(r)=\int_{|z|\le r}g_\Omega(z)\,dz
     =2\pi\int_0^r t\bar g_\Om(t)\,dt,
\end{equation}
which implies \eqref{Psi_Om'-bar_g}. 
Thus $\Psi_\Om\in C^2((0,\infty))$. 
The second equality of \eqref{Psi_nu-bar_g'} follows directly from \eqref{Psi_Om'-bar_g} and \eqref{eq:Psi_nu-bar_g}. 

Since $\Psi'_\Om=0$ for $r\ge\diam\Om$, integrating 
\[
\left(\frac{\Psi'_\Om(r)}{r}\right)'=-\frac{\Psi_{\nu;\Ga}(r)}r
\]
from $r$ to $\infty$ gives 
\[
\Psi'_\Om(r)=r\int_r^\infty \frac{\Psi_{\nu;\Ga}(t)}t\,\d t.
\]
Integrating from $0$ to $r$ and interchanging the order of integration gives \eqref{eq:inverse}. 
\end{proof}

\begin{theorem}\label{thm_equivalence}{\rm (\cite{M, Oball, OS2})} 
If $\Omega$ satisfies the condition $(\ast)$, then 
$\Psi_\Om$, $\Psi_{\nu;\,\Ga}$, $\bar g_\Om$, $B_\Om$, $B_{\nu;\,\Ga}$ are equivalent. 
Furthermore, if $\Omega$ is convex, $\Phi_\Om$ is equivalent to the above five. 
\end{theorem}

\begin{proof}
We introduce some known results which are necessary to show the equivalence between $\Psi_\Om, \Psi_{\nu;\Ga}$ and $B_\Om, B_{\nu;\Ga}$, and when $\Om$ is convex, $\Phi_\Om$. 

The equivalence between the IDD $\Psi_\Om$ and the chord length distribution $\Phi_\Om$ for a convex $\Omega$ was stated in \cite{M} p.25, and it is a consequence of the formula\footnote{This can be considered as a special case of the Blaschke-Petkantschin formula (\cite{B2,P}). } 
\begin{equation}\label{B-P_formula}
\d x\wedge \d y=|t_2-t_1|\,\d\mu \wedge \d t_1\wedge\d t_2 \quad (x,y\in\Omega), %\nonumber
\end{equation}
where $t_1$ and $t_2$ are the signed distance of $x$ and $y$ from the foot of the perpendicular to the line through $x$ and $y$ from the origin (\cite{San2} (4.2) p.46). 

The equivalence of the IDD $\Psi_\Om(r)$ and the Riesz energy function $B_\Om(z)$ follows from the Mellin-Stieltjes transform expression 
\begin{equation}\label{Mf}
B_\Om(s)
=\int_{[0,\infty)}r^s\,d\Psi_\Om(r),
\qquad (\R s>-2) \nonumber
\end{equation}
(cf. \cite{OS2} Proposition 3.3 and Subsection 4.2, and \cite{Oball} Prop 2.1).  

Finally, the equivalence of the Riesz energy function $B_\Om(z)$ and its $\nu$-weighted version of the boundary, $B_{\nu;\,\Ga}(z)$, is given by application of Stokes (Gauss) theorem twice: 
\begin{equation}\label{B_Omega_boundary_integral}
B_\Om(z)=
\frac{-1}{(z+2)^2}\iint_{\Ga\times \Ga}{|x-y|}^{z+2}\langle\uon{x},\uon{y}\rangle\,\d x \d y 
\nonumber
\end{equation}
(\cite{OS2}, Lemma 4.1). 
We remark that the above equation also holds when the boundary is a polygon since the Gauss (Stokes) theorem does not require that the boundary be smooth. 
\end{proof}

We remark that the above also holds for higher dimensions. 
In $\RR^d$ the same computation yields
\[
\Psi_{\nu;\Ga}(r)=(d-1)\frac{\Psi'_\Om(r)}{r}-\Psi''_\Om(r),
\qquad
\Psi'_\Om(r)=r^{\,d-1}\int_r^\infty\frac{\Psi_{\nu;\Ga}(t)}{t^{\,d-1}}\,dt,
\]
and 
\[
B_\Om(z)=
\frac{-1}{(z+2)(z+d)}\iint_{\partial\Om\times \partial\Om}{|x-y|}^{z+2}\langle\uon{x},\uon{y}\rangle\,\d x \d y. 
\]

\begin{remark}\label{rem:2}
The equality \eqref{eq:laplace} can be refined as follows. 
For $i,j=1,2$, let $\eta_{ij}$ be the push-forward of the signed measure $(\nu_x)_i(\nu_y)_j\,\d x\,\d y$ on $\Ga\times\Ga$ under the difference map $(x,y)\mapsto x-y$. 
Then, as a distribution, $\partial_i\partial_jg_\Om=-\eta_{ij}$. 
By taking the trace, we obtain \eqref{eq:laplace}. 
\end{remark}

\begin{remark}\label{rem_equiv2}\rm 
We introduce some further direct relation between the quantities. 

The equivalence of $\bar g_\Om$ and $\Phi_\Om$ for a convex body $\Om$ is given as follows 
(cf. \cite{AB} Remark 5.5 and \cite{Matheron} p. 86). 
If we put $u_\t=(\cos\t, \sin\t)$ and $s_+=\max\{s,0\}$, 
\[
g_\Om(ru_\theta)
=\int_{u_\theta^\perp}
\bigl(\sigma_\Om(\ell_{p,\t})-r\bigr)_+\,dp,
\]
where $\ell_{p,\t}$ is the line $\{p+su_\t\,:\,s\in\RR\}$. 
Since $g_\Om(z)=g_\Om(-z)$, integration in $\t$ gives 
\[
\pi\bar g_\Om(r)
=\int_0^\infty(t-r)_+\,d\Phi_\Om(t)
=\int_r^\infty(t-r)\,d\Phi_\Om(t).
\]
Since $\bar g_\Omega\in C^1((0,\infty))$ by Proposition \ref{prop:equiv_via_bar_g}, differentiation gives
\[
  \pi\bar g_\Omega'(r)
  =\Phi_\Omega(r)-\Phi_\Omega(\infty),
  \qquad r>0.
\]
In particular, $\Phi_\Omega$ is continuous, and 
\[
  \pi D^2\bar g_\Omega=d\Phi_\Omega
  \quad\text{on }(0,\infty)
\]
in the sense of distributions. 
Consequently,
\[
  \Phi_\Omega'(r)=\pi\bar g_\Omega''(r)
  \quad\text{for almost every }\> r>0.
\]

We remark that $\bar g_\Om''$ is called the chord-length distribution in \cite{C20}.

The equivalence of $B_\Om$ and $\Psi_\Om$ also follows from the Hausdorff moment problem. 
Since the distance measure $d\Psi_X$ is supported on the compact interval $[0,\operatorname{diam}X]$, its moments $B_X(k)$, $k=0,1,2,\ldots$, determine it uniquely. Hence $B_X$ and $\Psi_X$ determine each other.

The same moment argument applies to the finite signed measure $d\Psi_{\nu;\Gamma}$, yielding the equivalence between $B_{\nu;\Gamma}$ and $\Psi_{\nu;\Gamma}$. 
In fact, although $ d\Psi_{\nu;\Ga}$ is not necessarily a positive measure, since polynomials are dense in $C([0, \textrm{diam} X])$ by the Weierstrass approximation theorem, a finite (possibly signed) Borel measure on this interval is uniquely determined by its non-negative integer moments. 

The equivalence between $\Phi_\Omega$ and $B_\Omega$ for a convex body $\Om$ is given as follows. 
From \eqref{B-P_formula} we get 
\begin{equation}\label{Riesz-chord}
B_\Om(s)
=\frac2{(s+2)(s+3)}\int_{\mathcal{E}_1}{(\sigma_\Om(\ell))}^{s+3}\,\d\mu 
\end{equation}
for $s>-2$ (\cite{San2} p.47, (4.4))\footnote{The formula \eqref{Riesz-chord} is generalized to arbitrary dimension (\cite{San2} p.238, (14.25)), which is generalized to non-convex case (\cite{CB} (3.5)).}, which implies that the Riesz energy function can be expressed as a Lebesgue-Stieltjes integral as 
\[
B_\Omega(s)=\frac2{(s+2)(s+3)}\int_0^\infty \sigma^{s+3}\,\d\Phi_\Omega(\sigma) 
\]
(\cite{San2} p.46 (4.2)). 
\end{remark}

%!!!!!!!!!!!!!!!!!!!!!!!!!!!!!!!!!!!!!%%%%%%%%%%%%%%%%%%%%%%%%%%%%%%%%%%%%
\section{Analysis of the $\nu$-weighted IDD for polygons}
%!!!!!!!!!!!!!!!!!!!!!!!!!!!!!!!!!!!!!%%%%%%%%%%%%%%%%%%%%%%%%%%%%%%%%%%%%
%
The purpose of this section is to establish Corollary \ref{perimeter} and Propositions \ref{from_Psi_to_polygon} and \ref{Psi<->angles}, which are needed in Section \ref{sect_4} for the reconstruction of generic polygons. 
In Subsection \ref{subs3-2}, we treat the possible local configurations of pairs of sides, and show how the geometric data about lengths and angles are obtained from the $\nu$-weighted interpoint distance distribution of $\Ga$. 

In what follows we assume that $\Ga$ is a planar polygon. 
Recall that we assume that $\Ga$ is connected and simple closed, and that the interior angle at each vertex is neither $0$ nor $\pi$. 

Let $P_1,\dots,P_n$ be the vertices in counterclockwise order and $E_i=\overline{P_{i-1}P_i}$ be the sides. 
We study the contribution of a pair of sides $E_i$ and $E_j$ to the (normal weighted) interpoint distance distribution. 

We show that, under genericity condition that guarantees that the individual contributions to the singularity remain distinct and can be extracted separately, the values of $r$ where the second derivative $\Psi_{\nu;\Ga}$ is discontinuous give information on the intervertex lengths and lengths of perpendiculars from vertices to sides (Proposition \ref{from_Psi_to_polygon}), and the jumps or the singular part of the second and third derivatives there give information on angles (Proposition \ref{Psi<->angles}), 

where by singular part we mean: 
\begin{definition}\label{def:singular} \rm 
If, as $\rho\searrow0$,
\[
f(c+\rho)=\sum_{j=1}^{m}b_j\rho^{-\lambda_j}+O(1),
\qquad
\lambda_1>\cdots>\lambda_m>0,
\]
we call the sum on the right the {\sl singular part} of $f$ at $c$ from the right, and $b_1\rho^{-\lambda_1}$ its leading singular part.
\end{definition}

These form the key components of this paper. 
For each critical configuration, we compute the finite jumps and the divergent terms of the second and third derivatives of $\Psi_{\nu;\Gamma}$ by case analysis.

\begin{table}[htbp]
 \begin{center}
  \begin{tabular}{|l|l|l|}
\hline
 Data read from $\Psi_{\nu;\Gamma}$ & Extraction / Source Point  & Place \\
\hline
 Perimeter & $\frac12\Psi_{\nu;\Gamma}'(0^+)$ & Corollary \ref{perimeter} \\
 Critical lengths & points where $\Psi_{\nu;\Gamma}''$ fails to be $C^2$ & Proposition \ref{from_Psi_to_polygon} \\
 Tangent data at both ends of sides & Jumps and singular coefficients at side lengths & Proposition \ref{Psi<->angles} \\
 Cyclic order and signed angles & Matching the above data & Theorem \ref{thm_generic} \\
\hline
  \end{tabular}
  \caption{Reconstruction procedure}
  \label{table}
 \end{center}
\end{table}

We use the following notation. 
Let $a_i$ be the length of $E_i$, and $\theta_i$ ($-\pi<\t_i<\pi, \> \t_i\ne0$) be the signed exterior angles, i.e. the signed angle of $\overrightarrow{P_iP_{i+1}}$ from $\overrightarrow{P_{i-1}P_{i}}$ (cf. Figure \ref{generic_endpoints}). 
Let $\ell_i$ be the line containing $E_i$ and $\mbox{\rm pr}_{i}$ be the orthogonal projection to $\ell_i$. 
We will denote a jump in the derivative of $\Psi$ by the symbol $J$; 
\[
J^{(p)}_{(\nu;)\, Y,W}(c)=\displaystyle \lim_{r\searrow c}\Psi_{(\nu;)\,Y,W}^{(p)}(r)-\lim_{r\nearrow c}\Psi_{(\nu;)\,Y,W}^{(p)}(r) \qquad (Y,W\subset \Ga). 
\]
In particular, when $Y$ and $W$ are sides $E_i$ and $E_j$ we abbreviate 
\begin{equation}\label{def_J}\begin{array}{rcl}
J^{(p)}_{(\nu;)\, i,j}(c)&=&\displaystyle \lim_{r\searrow c}\Psi_{(\nu;)\,E_{i}, E_j}^{(p)}(r)-\lim_{r\nearrow c}\Psi_{(\nu;)\,E_{i}, E_j}^{(p)}(r). 
\end{array}
\end{equation}
For a single side $E_i$, direct computation shows 
\begin{equation}\label{Psi_single}
\Psi_{E_i}(r)%=\Psi_{E_i,E_i}(r)
=\left\{\begin{array}{ll}
2a_ir-r^2 &\quad (r<a_i),\\[1mm]
{a_i}^2 &\quad (r\ge a_i), 
\end{array}
\right.
\end{equation}
which implies 
\[
J^{(2)}_{i,i}(a_i)=J^{(2)}_{\nu;\,i,i}(a_i)=2.  
\]

\begin{remark}\label{rem:3-1}\rm 
Related edge-pair formulae for second-order distributional derivatives of the covariogram are known for polygonal sets; see Theorem 3.6 of \cite{Sch} and Lemma 3.9 of \cite{BBD}. 
Our problem differs in that the directional variable of the covariogram has been integrated out in the sense of \eqref{eq:Psi_nu-bar_g}. We therefore analyze the singularities that remain in the one-dimensional radial distribution $\Psi_{\nu;\Gamma}(r)$. 
\end{remark}

%!!!!!!!!!!!!!!!!!!!!!!!!!!!!!!!!!!!!!%%%%%%%%%%%%%%%%%%%%%%%%%%%%%%%%%%%%
\subsection{Contribution of adjacent edges at $r=0$}\label{subs3-1}
%!!!!!!!!!!!!!!!!!!!!!!!!!!!!!!!!!!!!!%%%%%%%%%%%%%%%%%%%%%%%%%%%%%%%%%%%%

\begin{definition}\label{def_omega} \rm 
Define $\omega(\t)$ and $\omega_\nu(\t)$ $(-\pi<\t<\pi)$ by
\[
\omega(\t)=\frac{\t}{\sin\t} \quad (\t\ne0), \quad
\omega(0)=1, \qquad
\omega_\nu(\t)={\t}{\cot\t} \quad (\t\ne0), \quad
\omega_\nu(0)=1.
\]

\end{definition}

Then $\omega$ and $\omega_\nu$ are differentiable and even, $\omega_{(\nu)}(\t)=\omega_{(\nu)}(-\t)$.

\begin{lemma}\label{lemma_half_IDD}
Put $a=\min\{a_i,a_{i+1}\}$. Then 
\begin{equation}\label{IDD_adjacent}
\Psi_{E_i,E_{i+1}}(r)=\frac12 \omega(\t_i)\,r^2 \quad 
\left\{\begin{array}{ll}
r\le a  &\quad \displaystyle \left(0\le|\t_i|<\frac\pi2\right), \\[3mm]
r\le a\sin|\t_i| &\quad \displaystyle \left(\frac\pi2\le|\t_i|<\pi\right).
\end{array}
\right.
\end{equation}
\end{lemma}

\begin{proof}
Put $\t=\t_i$. Since $\omega(\t)$ is even, we may assume $\t\ge0$. 

When $0\le\t<\pi/2$ 
\[
\begin{array}{rcl}
\Psi_{E_i,E_{i+1}}(r)&=&\displaystyle \int_0^r\left(-t\cos\t+\sqrt{r^2-t^2\sin^2\t}+t\right)\d t \\[4mm]
&=&\displaystyle -\frac{\cos\t}2 r^2+\frac{r^2}{\sin\t}\left[\frac{s \sqrt{1-s^2}}2+\frac12\arcsin s\right]_0^{\sin\t} \\[4mm]
&=&\displaystyle \frac{r^2}2\left(-\cos\t+\frac{\sin\t\cos\t+\t}{\sin\t}\right)  \\[4mm]
&=& \displaystyle \frac12 \,\omega(\t)\,r^2.
\end{array}
\]

When $\pi/2\le\t<\pi$ and $0\le r\le a\sin\t$ 
\[
\begin{array}{rcl}
\Psi_{E_i,E_{i+1}}(r)&=&\displaystyle \int_0^r\left(-t\cos\t+\sqrt{r^2-t^2\sin^2\t}+t\right)\d t +\int_r^{\frac{r}{\sin\t}}
2\sqrt{r^2-t^2\sin^2\t}\,\,\d t \\[4mm]
&=&\displaystyle -\frac{\cos\t}2 r^2+\frac{r^2}{\sin\t}\left[\frac{s \sqrt{1-s^2}}2+\frac12\arcsin s\right]_0^{\sin\t}
+ \frac{r^2}{\sin\t}\left[s \sqrt{1-s^2}+\arcsin s\right]_{\sin\t}^1 \\[4mm]
&=&\displaystyle -\frac{\cos\t}2 r^2+\frac{r^2}{\sin\t}\left[\frac12\left(-\sin\t\cos\t+\pi-\t\right)
+\left(\frac{\pi}2+\sin\t\cos\t-(\pi-\t)\right)\right] \\[4mm]
&=& \displaystyle \frac12 \omega(\t)\,r^2.
\end{array}
\]
\end{proof}

\begin{corollary}\label{Cor_Psi_r_small}
For sufficiently small $r\ge0$, we have
\begin{equation}\label{Psi_nu_r_small}
\Psi_{\nu;\Ga}(r)=2\operatorname{Length}(\Ga)r-nr^2+\sum_{i=1}^n \omega_\nu(\t_i) \,r^2 \quad (\Ga=\pOm).
\end{equation}
\end{corollary}

\begin{proof}
Put 
\begin{equation}\label{def:r0}
r_0=\min\Big\{\min_i a_i, \ \min_{|i-j|\ge2}\operatorname{dist}(E_i,E_j),\ \min_{|\theta_i|\ge\pi/2}\big\{\min\{a_i,a_{i+1}\}|\sin\theta_i|\big\}\Big\}.
\end{equation}
Since $\langle\nu_i,\nu_{i+1}\rangle=\cos\t_i$, if $0\le r<r_0$ we have 
\[\begin{array}{rcl}
\Psi_{\nu;\Ga}(r)&=&\displaystyle \sum_{i=1}^n \Psi_{E_i,E_{i}}(r)+2\sum_{i=1}^n \cos\t_i\Psi_{E_i,E_{i+1}}(r).
\end{array}
\]
The assertion follows from \eqref{Psi_single} and \eqref{IDD_adjacent}. 
\end{proof}

\begin{corollary}\label{perimeter}
The length of $\Ga$ can be obtained from $\Psi_{\nu;\,\Ga}$ as 
\[
\operatorname{Length}(\Ga)=\frac12 \Psi_{\nu;\,\Ga}'(0^+). 
\]
\end{corollary}

\begin{proof}
Immediate from \eqref{Psi_nu_r_small}
\end{proof}

%!!!!!!!!!!!!!!!!!!!!!!!!!!!!!!!!!!!!!%%%%%%%%%%%%%%%%%%%%%%%%%%%%%%%%%%%%
\subsection{Contribution of generic sides}\label{subs3-2} 
%!!!!!!!!!!!!!!!!!!!!!!!!!!!!!!!!!!!!!%%%%%%%%%%%%%%%%%%%%%%%%%%%%%%%%%%%%
%
We study the discontinuity of the second derivative of $\Psi_{E_i,E_{j}}$ at specific values. 

Let $P$ be an endpoint of $E_i$ and $P'$ be an endpoint of $E_j$. 
Let $b$ be the length of the segment $[{PP'}]$, 
and let $\a$ and $\beta$ be the signed exterior angles at $P$ and $P'$ (Figure \ref{generic_endpoints}). 

Let $\ell$ be the line through $P$ and $P'$. 
Let $X$ and $Y$ be points on $E_i$ and $E_j$ respectively. 
If $XP=s$ and $YP'=t$ then the distance between $X$ and $Y$, which we denote by $f(s,t)$, is given by 
\begin{equation}\label{distance_sq}
f(s,t)=\sqrt{b^2+s^2+t^2+2bs\cos\a+2bt\cos\beta+2st\cos(\a+\beta)} \,.
\end{equation}

\begin{figure}[htbp]  

\begin{center}
\begin{minipage}{.6\linewidth}
\begin{center}
  \includegraphics[width=0.75\linewidth]{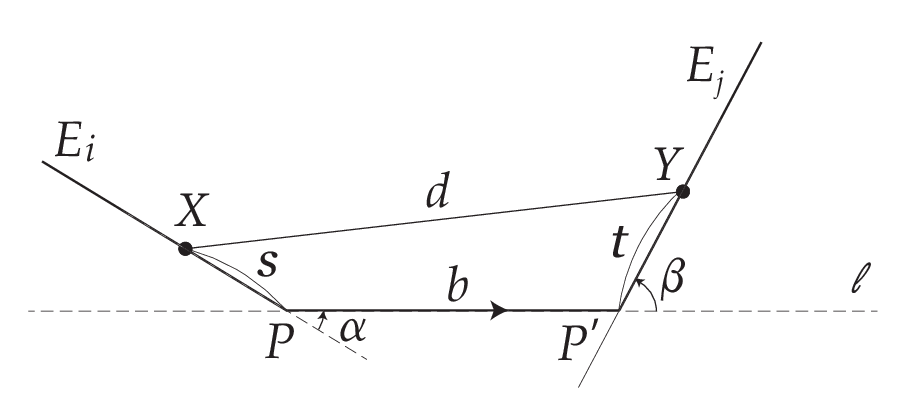}
  \caption{A configuration of a pair of generic sides. $d=f(s,t)$.}
  \label{generic_endpoints}
\end{center}
\end{minipage}
\hskip 0.2cm
\begin{minipage}{.37\linewidth}
\begin{center}

\usetikzlibrary{arrows.meta, calc}

\begin{tikzpicture}[scale=0.88]

\def\W{3.0}
\def\H{3.4}

\pgfmathsetmacro{\cx}{2.8}
\pgfmathsetmacro{\cy}{3.3}
\pgfmathsetmacro{\rx}{3.0}
\pgfmathsetmacro{\ry}{3.5}

\pgfmathsetmacro{\tL}{180 + acos(\cx/\rx)} % left side intersection
\pgfmathsetmacro{\tB}{180 + asin(\cy/\ry)} % bottom side intersection

\coordinate (O)  at (0,0);
\coordinate (PL) at ({\cx+\rx*cos(\tL)},{\cy+\ry*sin(\tL)});
\coordinate (PB) at ({\cx+\rx*cos(\tB)},{\cy+\ry*sin(\tB)});

\fill[green!25]
  (O) -- (PL)
  plot[domain=\tL:\tB,smooth,variable=\t]
    ({\cx+\rx*cos(\t)},{\cy+\ry*sin(\t)})
  -- (O) -- cycle;

\draw[thick] (0,0) rectangle (\W,\H);

\draw[very thick]
  plot[domain=\tL:\tB,smooth,variable=\t]
    ({\cx+\rx*cos(\t)},{\cy+\ry*sin(\t)});

\node[left]  at (0,1.8) {$E_j$};
\node[below] at (1.6,0) {$E_i$};

\node at (0.8,1.5) {$L_r$};
\node at (0.5,0.5) {$S_r$};

\end{tikzpicture}

\caption{The sublevel set $S_r$ (shaded) and the level set $L_r$ in $R$}
\label{fig:level_set}
\end{center}
\end{minipage}
\end{center}

\end{figure}

Let $S_r$ be the sublevel set of $f$, $S_r=\{(s,t)\,:\,f(s,t)\le r\}$ ($r>0$). 
Then $\Psi_{E_i,E_{j}}(r)$ is equal to the area of $S_r\cap([0,a_i]\times[0,a_j])$, where $a_i$ and $a_j$ are the lengths of $E_i$ and $E_j$. 
The level set $L_r=\{(s,t)\,:\,f(s,t)=r\}$ is (a part of) an ellipse in general, although it may degenerate to a point or a line (segment), where the latter case can occur only when $E_i$ and $E_j$ are parallel (i.e. if $\a+\beta=0, \pm\pi$ ). 

Put $R=[0,a_i]\times[0,a_j]$. 
Then $\Psi_{E_i,E_{j}}$ is smooth in a neighbourhood of $r$ 
if the level set $L_r$ either does not intersect the boundary of $R$ or intersects it transversely at every point that is not a corner of $\partial R$. 
Namely, the smoothness of $\Psi_{E_i,E_{j}}$ may fail at $r$ when the level set $L_r$ either passes through a corner of $R$ or is tangent to the boundary of $R$. 
The former is the case in which $r$ is the distance between endpoints of edges, and the latter is the case in which $r$ is the height of a vertex from a side, i.e. the length of the perpendicular from the vertex to the side when the foot of the perpendicular belongs to the side. 

\begin{definition}\label{multiset} \rm 
Put $H_{ji}=\mbox{pr}_i(P_j)$. 
Define the multiset (i.e. a set that allows multiple instances for each of its elements) of {\em critical lengths} by 
\begin{equation}\label{jump_multiset}
\mathcal{C}=[P_iP_j\,:\,j>i\,]\cup [P_jH_{ji}\,:\,H_{ji}\in \mbox{\rm Int} E_i], 
\end{equation}
where $\mbox{\rm Int}E_i$ is the interior of a side $E_i$. 
We call $P_iP_j$ the {\em intervertex length} ({\em side length} when $|i-j|=1$ and {\em diagonal length} otherwise) and $P_jH_{ji}$ $(H_{ji}\in \mbox{\rm Int} E_i)$ the {\em interior height}. 
Let $c_0$ denote the smallest critical length. 
\end{definition}
We remark $c_0$ is equal to $r_0$ given by \eqref{def:r0}. 
\begin{proposition}\label{prop_smoothness_outside_C}
$\Psi_{E_i,E_{j}}$ and $\Psi_{\nu;\,E_i,E_{j}}$ are smooth at $l$ if $l\not\in\mathcal{C}$. 
\end{proposition}

We remark that if a critical value of the distance function is given by a pair of interior points of sides then two sides must be parallel and the value is also given as either an intervertex length or an interior height. 

We will see that from the jumps of the values or the singular part of the second and third derivatives we obtain information on angles.

%!!!!!!!!!!!!!!!!!!!!!!!!!!!!!!!!!!!!!%%%%%%%%%%%%%%%%%%%%%%%%%%%%%%%%%%%%
\subsubsection{The case when the level set passes through a vertex}
%!!!!!!!!!!!!!!!!!!!!!!!!!!!!!!!!!!!!!%%%%%%%%%%%%%%%%%%%%%%%%%%%%%%%%%%%%

We begin with the case when the level set $L_b$ passes through exactly one vertex of $R=[0,a_i]\times[0,a_j]$, but is not tangent to the boundary there. 
Let $\mathcal{C}_{E_i,E_{j}}$ be a submultiset of $\mathcal{C}$ coming from a pair $E_i$ and $E_j$, i.e. 
$$\begin{array}{rcl}
\mathcal{C}_{E_i,E_{j}}=[P_kP_\la:k\in\{i-1,i\},\,\la\in\{j-1,j\}, \, k\ne \la]\cup[P_kH_{k\la}:H_{k\la}\in \mbox{\rm Int} E_\la, \,\{k',\la\}=\{i,j\}, \,  k\in\{k'-1,k'\}]. 
\end{array}$$

\begin{proposition}\label{lem_jump_Psi_endpoints}
Assume $\a,\beta\ne\pm\pi/2$. 
Assume $b$ appears only once in the submultiset $\mathcal{C}_{E_i,E_{j}}$. 
Then the jumps of $\Psi_{\nu;E_i,E_{j}}''$ and $\Psi_{\nu;E_i,E_{j}}'''$ at $r=b$ are given by 
\begin{eqnarray}
\displaystyle J^{(1)}_{\nu;\, i,j}(b)
&=&\displaystyle 0, \nonumber \\[0mm]
\displaystyle J^{(2)}_{\nu;\, i,j}(b)
&=&\displaystyle \pm(1-\tan\a\cdot\tan\beta), \label{jump_nu} \\[0mm]
\displaystyle J^{(3)}_{\nu;\, i,j}(b)
&=&\displaystyle \pm(1-\tan\a\cdot\tan\beta)\left(\tan\a\cdot\tan\beta-\tan^2\a-\tan^2\beta\right)\frac1b, \label{jump_nu_3} 
\end{eqnarray}
where the signs are positive when the orientations of the two sides $E_i$ and $E_j$ both match the orientation of $\overrightarrow{PP'}$ or both do not, and negative otherwise. 
On the other hand, $\Psi_{\nu;E_i,E_{j}}'$ does not have jumps at $r=b$. 

We remark that the formulae hold even if $\a,\beta=0$. 
\end{proposition}

\begin{proof}
We show 
\begin{eqnarray}
\displaystyle J^{(1)}_{i,j}(b)
&=&\displaystyle 0, \label{jump0} \\[0mm]
\displaystyle J^{(2)}_{i,j}(b)
&=&\displaystyle \frac1{\cos\a\,\cos\beta}, \label{jump} \\[0mm]
\displaystyle J^{(3)}_{i,j}(b)
&=&\displaystyle \frac1{\cos\a\,\cos\beta}\left(\tan\a\cdot\tan\beta-\tan^2\a-\tan^2\beta\right)\frac1b. \label{jump3} 
\end{eqnarray}
Then the assertion follows since $\langle\nu_i,\nu_j\rangle=\pm\cos(\a+\beta)$, where the sign depends on the orientations of $E_i$ and $E_j$ with respect to the orientation of $PP'$. 

\smallskip
Step 1. First, assume $0<|\a|,|\beta|<\pi/2$. 
When $r\ge b$ and $r-b\ll 1$, for $t\ge0$, there exists $s\ge0$ that satisfies $f(s,t)\le r$ if and only if 
\begin{equation}\label{range_of_t}
0\le t\le -b\cos\beta+\sqrt{r^2-b^2\sin^2\beta}. 
\end{equation}
We expand $\Psi_{E_i,E_{j}}(r)$ in a series in \[\rho=r-b.\] 

The level set $L_r$ passes through the origin when $r=b$. Since $b$ appears only once in the multiset $\mathcal{C}_{E_i,E_{j}}$, the level set $L_b$ does not pass through the other vertices of $R=[0,a_i]\times[0,a_j]$ and does not touch the boundary of $R$. 
Therefore the jumps in the second and third derivatives come from the contribution of a neighbourhood of the origin. 
The gradient of $f$ at the origin is given by 
\[
\nabla f(0,0)=(\cos\a, \cos\beta). 
\]

(i) To avoid the degeneracy in the following proof, we further assume that $\a+\beta\ne0$ here, so that $E_i$ and $E_j$ are not parallel. 
Then $\nabla f(0,0)$ is in the first quadrant $R_{++}$. 
We assume $b=1$ for simplicity's sake in what follows. 
This is possible since, under a dilation by a factor $\la>0$, we have 
$b\to\la b, \> r\to\la r, \> \Psi \to \la^2\Psi$, which implies that $J^{(2)}$ is unchanged and $J^{(3)}\to (1/\la)J^{(3)}$. 
Let $s_0$ and $t_0$ be the intercepts of the level set $L_r$ with the $s$- and $t$-axes, and put $Q_1=(s_0,0)$ and $Q_2=(0,t_0)$. 
\begin{figure}[htbp]
\begin{center}
\includegraphics[width=.27\linewidth]{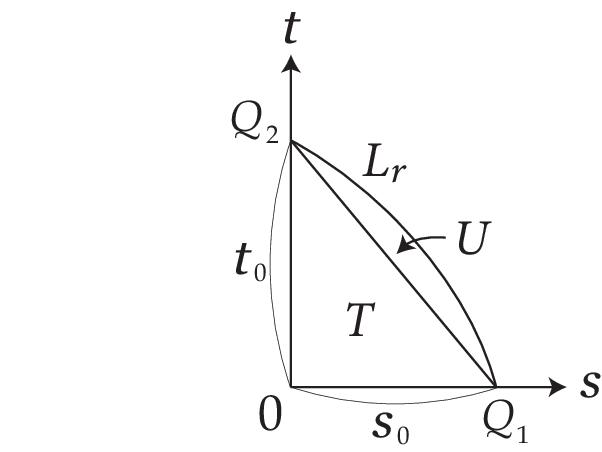}
\caption{The level set $L_r$ in the $(s,t)$ plane}
\label{TU}
\end{center}
\end{figure}
Let $T$ be the triangle $\triangle OQ_1Q_2$ and $U$ $(U\subset R)$ be a region bounded by the level set $L_r$ and the line segment $[Q_1Q_2]$ (Figure \ref{TU}). 
Then $S_r\cap R=T\cup U$, so that $\Psi_{E_i,E_j}(r)=\mathrm{Vol}(T)+\mathrm{Vol}(U)$, where $S_r$ is the sublevel set of $f$. 
Since \eqref{range_of_t} indicates 
\[
s_0=\frac1{\cos\a}\,\rho-\frac{\tan^2\a}{2\cos\a}\,\rho^2+O(\rho^3), \quad
t_0=\frac1{\cos\beta}\,\rho-\frac{\tan^2\beta}{2\cos\beta}\,\rho^2+O(\rho^3),
\]
the area of the triangle $T$ satisfies 
\begin{equation}\label{Area_T}
\operatorname{Vol}(T)=\frac1{2\cos\a\cos\beta}\,\rho^2-\frac{\tan^2\a+\tan^2\beta}{4\cos\a\cos\beta}\,\rho^3+O(\rho^4). 
\end{equation}

We estimate the area of $U$. 
Remark that, under the assumption that $\a+\beta\ne0$, \eqref{distance_sq} indicates that the level set $L_r$ is an ellipse that can be expressed as 
\[
\frac{(u-u_0)^2}{A^2}+\frac{(v-v_0)^2}{B^2}=1, \quad u=\frac{s+t}{\sqrt2}, \> v=\frac{-s+t}{\sqrt2},
\]
where 
\[
A=\frac{r}{\sqrt{1+\cos(\a+\beta)}}, \> B=\frac{r}{\sqrt{1-\cos(\a+\beta)}}, 
\> u_0=-\frac{\cos\a+\cos\beta}{\sqrt2\,(1+\cos(\a+\beta))}, 
\> v_0=\frac{\cos\a-\cos\beta}{\sqrt2\,(1-\cos(\a+\beta))}. 
\]
Let 
\[
\xi(u,v)
=
\left(\frac{u-u_0}{A},\frac{v-v_0}{B}\right)
\]
be an affine map that maps $L_r$ to the unit circle. 
Since 
\[|\xi(Q_1)-\xi(Q_2)|=|\tan\a+\tan\beta|\rho+O(\rho^2)\]
and the area of a region bounded by the unit circle and a chord with length $\sigma$ 
is given by $\sigma^3/12+O(\sigma^4)$, 
\begin{eqnarray}
\operatorname{Vol}(U)&=&AB\cdot\frac1{12}|\tan\a+\tan\beta|^3\rho^3 + O(\rho^4) 
\nonumber \\
&=&\frac{(\tan\a+\tan\beta)^2}{12\cos\a\cos\beta}\rho^3+O(\rho^4). \label{Area_U}
\end{eqnarray}
By \eqref{Area_T} and \eqref{Area_U} we have 
\[
\Psi_{E_i,E_{j}}(r)=\frac1{2\cos\a\cos\beta}\,\rho^2+\frac{\tan\a\cdot\tan\beta-\tan^2\a-\tan^2\beta}{6\cos\a\cos\beta}\,\rho^3+O(\rho^4),
\]
which implies \eqref{jump} and \eqref{jump3}. 

\smallskip
(ii) When $\a+\beta=0$, the level set $L_r$ degenerates to parallel lines, so that $U=\emptyset$ and $\operatorname{Vol}(U)=0$; since \eqref{Area_T} does not depend on whether $\a+\beta=0$ or not, the expansion above remains valid.

\smallskip
Step 2. Next we show that the cases when $|\a|>\pi/2$ or $|\beta|>\pi/2$ can be obtained algebraically. 
Let $J^{(p)}(\a,\beta)$ $(|\a|,|\beta|\ne\pi/2, p=2,3)$ be the jump of $\Psi_{E_i,E_j}^{(p)}(r)$ at $r=b$. 
Let $\widetilde{E}_j$ $(\widetilde{E}_j\ne E_j)$ be an edge in the line containing $E_j$ that shares the endpoint $P'$ and is on the opposite side to $E_j$ (Figure \ref{generic_endpoints_tilde}). 
\begin{figure}[htbp]
\begin{center}
\includegraphics[width=.25\linewidth]{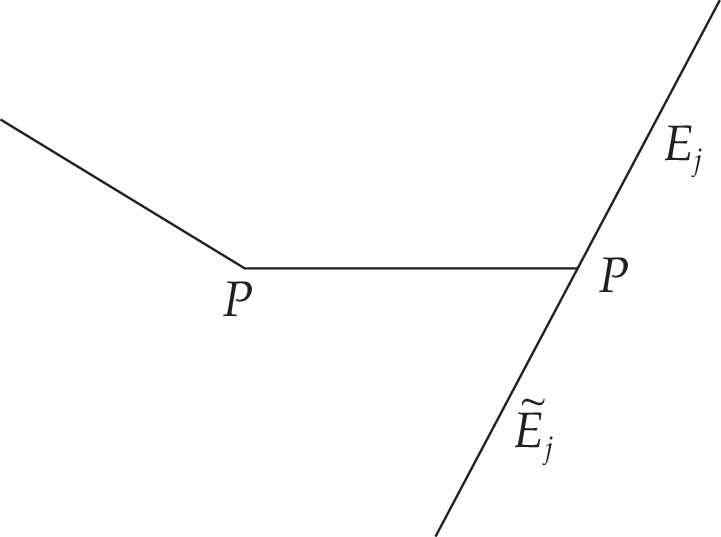}
\caption{$E_j$ and $\widetilde{E}_j$}
\label{generic_endpoints_tilde}
\end{center}
\end{figure}
Then $\Psi_{E_i, E_j\cup\widetilde{E}_j}$ is smooth at $r=b$. 
This is because the smoothness fails if and only if $b$ is the distance between $P$ and $E_j\cup\widetilde{E}_j$, which can occur if and only if $\beta=\pm \pi/2$. 
It follows that 
\begin{equation}\label{formula_J_algebraic}
J^{(p)}(\a,\beta)=-J^{(p)}(\a,\beta\pm\pi)=-J^{(p)}(\a\pm\pi,\beta)=J^{(p)}(-\a,-\beta) \qquad (p=2,3), 
\end{equation}
where the second equality comes from the symmetry in $\a$ and $\beta$, and the third from the invariance of \eqref{jump} and \eqref{jump3} under $(\alpha,\beta)\mapsto(-\alpha,-\beta)$, which corresponds to a reflection. 
Now \eqref{jump} and \eqref{jump3} for any pattern of $(\a,\beta)$, where \[\a,\beta\in (-\pi,-\pi/2)\cup(-\pi/2,0)\cup(0,\pi/2)\cup(\pi/2,\pi),\]
can be obtained by applying \eqref{formula_J_algebraic} 
to the case when $0<\a,|\beta|<\pi/2$ (Figure \ref{generic_endpoints_algebraic}). 
\begin{figure}[htbp]
\begin{center}
\includegraphics[width=.8\linewidth]{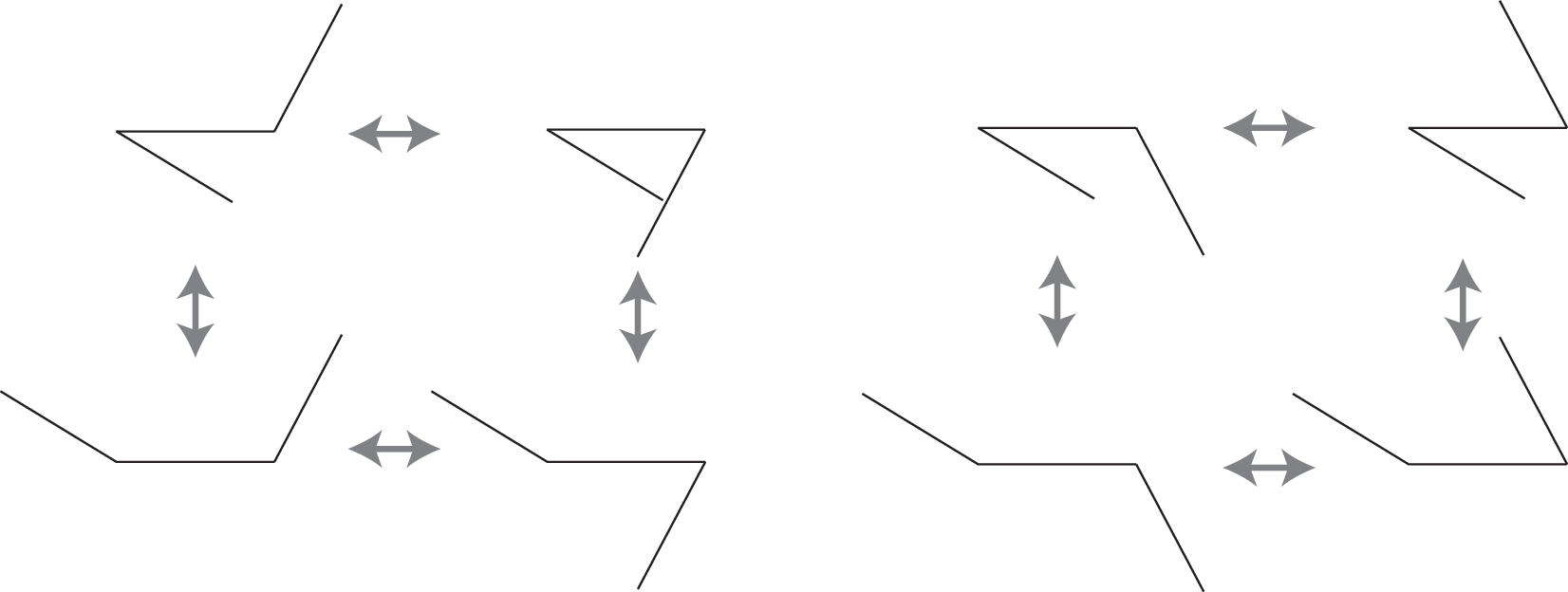}
\caption{Eight of the sixteen sign patterns of $(\alpha,\beta)$, according to which of the four intervals $(-\pi,-\pi/2),(-\pi/2,0),(0,\pi/2),(\pi/2,\pi)$ each angle lies in. Horizontal arrows indicate $\beta\mapsto\beta\pm\pi$ and vertical arrows $\alpha\mapsto\alpha\pm\pi$. By \eqref{formula_J_algebraic} each reverses the sign of $J^{(p)}$. The bottom-left panel of the left block is the base case $0<\alpha<\pi/2$, $0<\beta<\pi/2$.}
\label{generic_endpoints_algebraic}
\end{center}
\end{figure}

Finally, the cases $\a=0$ or $\beta=0$ follow by continuity and taking limits. 
\end{proof}

\begin{remark}\rm 
The formulae can also be verified by direct computation. 
For $0<|\a|,|\beta|<\pi/2$ and $r\ge b$ and $r-b\ll 1$, the region $\{f(s,t)\le r\}\cap R$ is described by \eqref{range_of_t} and 
\begin{equation}\label{}
0\le s \le -b\cos\a-t\cos(\a+\beta)+\sqrt{-t^2\sin^2(\a+\beta)+2bt(\cos\a\cos(\a+\beta)-\cos\beta)+r^2-b^2\sin^2\a}\, ,\nonumber
\end{equation}
which gives $\Psi_{E_i,E_{j}}(r)$ as an explicit but lengthy elementary function. 
Differentiating it three times and letting $r\searrow b$ yields \eqref{jump0}, \eqref{jump} and \eqref{jump3}. 
The author used Maple \cite{Maple} for this computation. 
\end{remark}

We remark that if exactly one of $\a$ and $\beta$ is equal to $\pm\pi/2$ then the level set $L_b$ is tangent to $\partial R$ at the origin. 

\begin{corollary}\label{cor_jump_adjacent}
Consider adjacent sides $E_i$ and $E_{i+1}$. 
Assume $a_i, a_{i+1}$ and $P_{i-1}P_{i+1}$ are all distinct. 
Assume $\t_i\ne\pm\pi/2$. Assume also that $a_i$ occurs only once in $\mathcal C_{E_i,E_{i+1}}$. 
If we let $J^{(2)}_\nu(\a,\beta)$ and $J^{(3)}_\nu(\a,\beta)$ denote the right hand sides of \eqref{jump_nu} and \eqref{jump_nu_3} with $b$ in \eqref{jump_nu_3} being replaced by $a_i$, the jumps of the derivatives of $\Psi_{\nu;E_i,E_{i+1}}(r)$ at $r=a_i$ are given by 
\begin{equation}
J^{(p)}_{\nu;{i},{i+1}}(a_i)=-J^{(p)}_{\nu}(0,\t_i). 
\end{equation}
\end{corollary}

\begin{proof}
Let $\widetilde{E}_i$ be the extension of side $E_i$ to the opposite side of $E_{i+1}$ (Figure \ref{Jump_a=0}). 
\begin{figure}[htbp]
\begin{center}
\includegraphics[width=.5\linewidth]{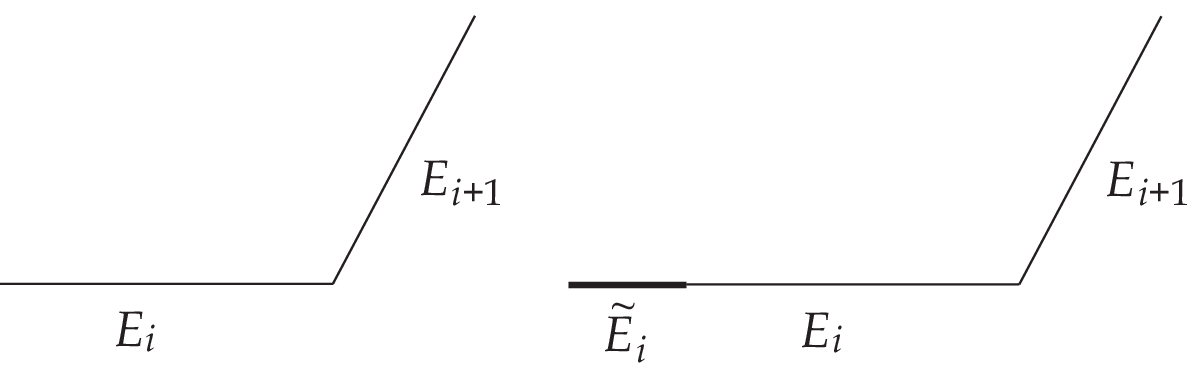}
\caption{Extension of a side}
\label{Jump_a=0}
\end{center}
\end{figure}
Then $\Psi_{\nu;\widetilde{E}_i\cup E_i,E_{i+1}}(r)$ is smooth at $a_i$. Therefore, 
\[
0=J^{(p)}_{\nu;\widetilde{E}_i\cup E_i,E_{i+1}}(a_i)
=J^{(p)}_{\nu;\widetilde{E}_i,E_{i+1}}(a_i)+J^{(p)}_{\nu;E_i,E_{i+1}}(a_i). 
\]
On the other hand, $J^{(p)}_{\nu;\widetilde{E}_i,E_{i+1}}(a_i)=J^{(p)}_{\nu}(0,\t_i)$, which implies the conclusion. 
\end{proof}

\begin{corollary}\label{cor_jumps_to_angles}
Consider three consecutive sides $E_{i-1}, E_i$ and $E_{i+1}$. 
Assume $\t_{i-1}, \t_i\ne\pm\pi/2$, and $a_i$ % 
appears only once in the sub-multiset of $\mathcal{C}$ obtained from $E_{i-1}\cup E_i\cup E_{i+1}$. 
Then we have 
\begin{eqnarray}
\displaystyle J^{(2)}_{\nu;\,E_{i-1}\cup E_i\cup E_{i+1}}(a_i)
&=& \displaystyle -2\tan\t_{i-1}\tan\t_i,  \label{jump_Psi_E0E1E2_nu''} \\[2mm]
\displaystyle J^{(3)}_{\nu;\,E_{i-1}\cup E_i\cup E_{i+1}}(a_i)
&=& \displaystyle \frac2{a_i}\tan\t_{i-1}\tan\t_i\left(\tan^2\t_{i-1}-\tan\t_{i-1}\tan\t_i+\tan^2\t_i+1\right). \quad{\phantom{a}}  \label{jump_Psi_E0E1E2_nu'''} 
\end{eqnarray}
\end{corollary}

\begin{proof}
The two equalities can be proved in the same way. 
Put $J^{(p)}_{\nu;\,k,l}=J^{(p)}_{\nu;\,k,l}(a_i)$ for $p=2,3$ and $k,l\in\{i-1,i,i+1\}$ (cf. \eqref{def_J}). 
Then, since $a_{i-1},a_{i+1}\ne a_i$ by the assumption, $J^{(p)}_{\nu;\,E_{i-1}\cup E_i\cup E_{i+1}}(a_i)$ $(p=2,3)$ 
can be expressed as 
\begin{equation}\label{f_J_1}
J^{(p)}_{\nu;\,i,i}+2\left(J^{(p)}_{\nu;\,{i-1},{i+1}}+J^{(p)}_{\nu;\,{i},{i-1}}+J^{(p)}_{\nu;\,{i},{i+1}}\right) \qquad p=2,3. 
\end{equation}
Let $J^{(2)}_\nu(\a,\beta)$ and $J^{(3)}_\nu(\a,\beta)$ be the right hand sides of \eqref{jump_nu} and \eqref{jump_nu_3} with $b$ being replaced by $a_i$ as in Corollary \ref{cor_jump_adjacent}. 
Then by definition 
\begin{equation}\label{f_J_2}
J^{(p)}_{\nu;\,{i-1},{i+1}}=J^{(p)}_{\nu}(\t_{i-1},\t_i).
\end{equation}
By Corollary \ref{cor_jump_adjacent} we have 
\begin{equation}\label{f_J_3}
J^{(p)}_{\nu;\,{i},{i-1}}(a_i)=-J^{(p)}_{\nu}(\t_{i-1},0), \quad
J^{(p)}_{\nu;\,{i},{i+1}}(a_i)=-J^{(p)}_{\nu}(0,\t_i),
\end{equation}
and \eqref{Psi_single} implies 
\begin{equation}\label{Jii}
J^{(2)}_{\nu;\, i,i}=2, \quad J^{(3)}_{\nu;\, i,i}=0. 
\end{equation}
The corollary's assertions follow from applying \eqref{jump_nu} and \eqref{jump_nu_3} to \eqref{f_J_2} and \eqref{f_J_3}, \eqref{Jii} and \eqref{f_J_1}. 
\end{proof}

Let $[XY]$ denote the line segment with endpoints $X$ and $Y$. 

\begin{corollary}\label{cor_jump_diagonal}
Suppose $A_1=[{P'P}]$ and $A_2=[{PP''}]$ (or $B_1=[{Q'Q}]$ and $B_2=[{QQ''}]$) are adjacent sides with common vertex $P$ (or $Q$ respectively). 
Let $\a_1,\a_2,\beta_1$ and $\beta_2$ be the signed exterior angles of $P'PQ, P''PQ$ at $P$ and $PQQ',PQQ''$ at $Q$ respectively (Figure \ref{generic_diagonal}). 
\begin{figure}[htbp]
\begin{center}
\includegraphics[width=.9\linewidth]{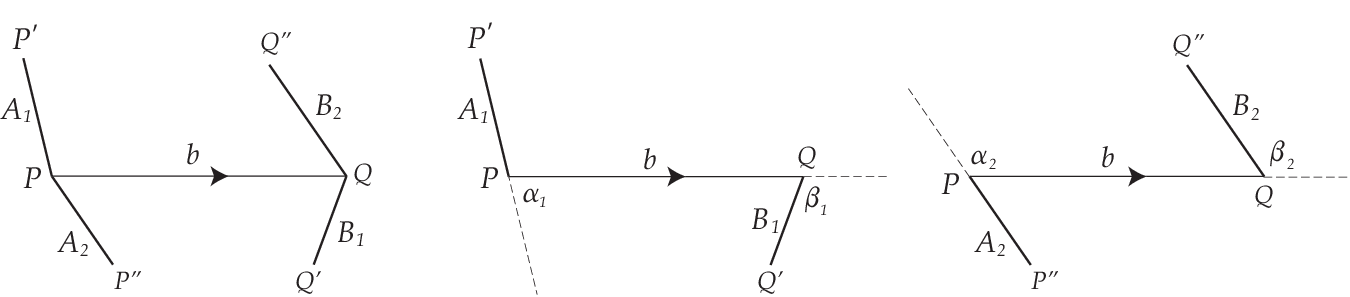}
\caption{In the case $\a_1>0, \beta_1<0, \a_2<0, \beta_2>0$ }
\label{generic_diagonal}
\end{center}
\end{figure}
Assume that none of $\a_1,\a_2,\beta_1$ and $\beta_2$ is equal to $\pm\pi/2$ and that the cyclic order of the sides is $A_1A_2B_1B_2$ or the reverse. 
If $b=PQ$ appears only once in the sub-multiset of $\mathcal{C}$ obtained from the pair $A_1\cup A_2$ and $B_1\cup B_2$, then the jumps of the second derivatives of $\Psi_{\nu;\,A_1\cup A_2, B_1\cup B_2}(r)$ at $b$ are given by 
\begin{eqnarray}
J^{(2)}_{\nu;\,A_1\cup A_2, B_1\cup B_2}(b)
&=&\displaystyle \left(\tan\a_1-\tan\a_2\right)\left(\tan\beta_1-\tan\beta_2\right). \label{jump_A1A2B1B2_nu} %\pm
\end{eqnarray}
\end{corollary}

\begin{proof}
\eqref{jump_A1A2B1B2_nu} follows from \eqref{jump_nu}. 
\end{proof}

Remark that $J^{(2)}_{\nu;\,A_1\cup A_2, B_1\cup B_2}(b)\ne0$. 
This is because $\tan\a_1=\tan\a_2$ if and only if $\a_1=\a_2$ modulo $\pi$, which can occur if and only if $\angle P'PP''=0, \pi$. However, it is impossible by our assumption that no two adjacent sides are collinear.

%!!!!!!!!!!!!!!!!!!!!!!!!!!!!!!!!!!!!!%%%%%%%%%%%%%%%%%%%%%%%%%%%%%%%%%%%%
\subsubsection{The case when the level set is tangent}
%!!!!!!!!!!!!!!!!!!!!!!!!!!!!!!!!!!!!!%%%%%%%%%%%%%%%%%%%%%%%%%%%%%%%%%%%%

Next we consider the case when the level set $L_r$ is tangent to the boundary of $R$. 
First assume that a tangent point is a vertex of $R$, which can occur if and only if exactly one of $\a$ and $\beta$ is equal to $\pm\pi/2$. %$\a=\pm\pi/2$ or $\beta=\pm\pi/2$. 

\begin{lemma}\label{lemma_tangency_vertex}
Assume $\a=\pi/2$ and $\beta\ne\pm\pi/2$. 
Denote the length of $PP'$ by $h$ (Figure \ref{vertex-tangency}). 
Put $\Psi_\beta=\Psi_{E_i,E_j}$. 
Assume that no other critical configuration for the pair $(E_i,E_j)$ occurs at distance $h$. 
\begin{enumerate}
\item The second derivative $\Psi_\beta''(r)$ satisfies 
\begin{eqnarray}\label{}
\Psi_\beta''(r)&=& O(1) \quad (r\nearrow h), \label{Psi''_beta_up} \\
\Psi_\beta''(r)&=& 
\displaystyle \frac{h}{\cos\beta \cdot \sqrt{r^2-h^2}}+O(1)
=\frac1{\cos\beta}\cdot\sqrt{\frac h2}\cdot\frac1{\sqrt\rho}+O(1)
\qquad  (\rho=r-h)\quad (r\searrow h), \label{Psi''_beta_down_blowup}
\end{eqnarray}

and therefore $J^{(2)}_{i,j}(h)
=\pm\infty$. 
Moreover, $\Psi_\beta$ is $C^1$ in a neighborhood of $h$.
\item Let $\gamma\equiv\pm(\beta+\pi)\pmod{2\pi}$, where
$\pi/2<|\beta|<\pi$. Then the leading singular terms cancel: 
\[
  \Psi_\beta''(r)+\Psi_\gamma''(r)=O(1)
  \qquad (r\searrow h).
\]

\item The singular part of $\Psi_{\nu;\,E_i,E_j}''$ as $r$ decreases to $h$ is given by 
\begin{equation}\label{eq:3-27}
\Psi_{\nu;\,E_i,E_j}''(r)=%\stackrel{r\searrow h}{\sim} 
\mp\,{\tan\beta}\cdot\sqrt{\frac h2}\cdot\frac1{\sqrt\rho}+O(1) \qquad (\rho=r-h),  
\end{equation}
where the sign is negative when the orientations of the two sides $E_i$ and $E_j$ both match the orientation of $\overrightarrow{PP'}$ or both do not, and positive otherwise. 
\end{enumerate}
\end{lemma}

\begin{proof}
Throughout this proof, we write $f\circeq g$ if $f-g$ extends to a smooth function on an open interval containing $h$, namely, $\circeq$ denotes equality modulo smooth functions near $r=h$. 
By the assumption on the critical configurations, the contribution away from $(P,P')$ is smooth in $r$ near $h$. 
If we truncate or extend the sides $E_i, E_j$ away from the endpoints $P,P'$ in such a way that no additional critical configuration at distance $h$ is created, $\Psi_\beta$ near $r=h$ may change only by a smooth function, which does not affect the assertions of the lemma. 
Therefore it is enough to consider local contribution of the sides with specific lengths. 

There are four cases as illustrated in Figure \ref{vertex-tangency}. 
\begin{figure}[htbp]
\begin{center}
\includegraphics[width=.7\linewidth]{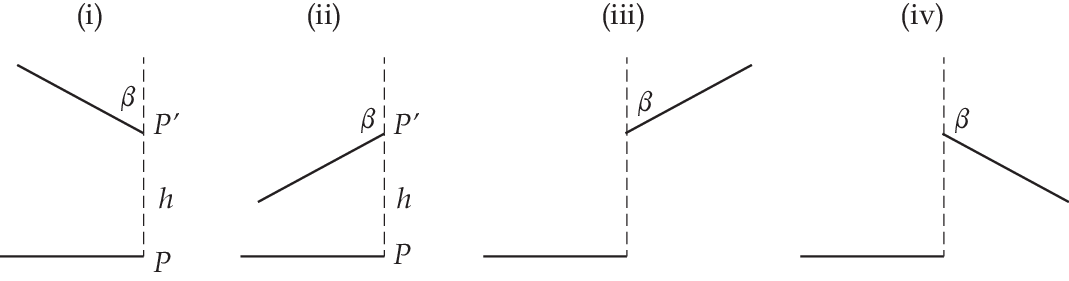}
\caption{(i) $0<\beta<\pi/2$, (ii) $\pi/2<\beta<\pi$, (iii) $-\pi/2<\beta<0$, (iv) $-\pi<\beta<-\pi/2$.}
\label{vertex-tangency}
\end{center}
\end{figure}

(1) (i) Suppose $0\le\beta<\pi/2$. Then for $r>h$
\begin{eqnarray}
\Psi_\beta(r)&\circeq&\displaystyle \int_0^{\sqrt{r^2-h^2\sin^2\beta}-h\cos\beta}\left(\sqrt{r^2-(h+t\cos\beta)^2}+t\sin\beta\right)\d t \nonumber\\[0mm]
&&\displaystyle +\int_{\sqrt{r^2-h^2\sin^2\beta}-h\cos\beta}^{(r-h)/\cos\beta}2\sqrt{r^2-(h+t\cos\beta)^2}\,\d t \label{term_vanish_beta=0}\\[1mm]
&=&\displaystyle \frac{1}{2\cos\beta}\left[r^2\left(\pi-\arctan\left(\frac{h\sin\beta+\sqrt{r^2-h^2\sin^2\beta}\,\cot\beta}{\sqrt{r^2-h^2\sin^2\beta}-h\cos\beta}\right)-\arctan\frac{h}{\sqrt{r^2-h^2}}\right) \right.\nonumber \\[1mm] 
&&\displaystyle \phantom{\frac1{2\cos\beta}+} \left.
-h\sin\beta\left(\sqrt{r^2-h^2\sin^2\beta}-h\cos\beta\right)-h\sqrt{r^2-h^2}\,\right]. \nonumber
\end{eqnarray}
Note that \eqref{term_vanish_beta=0} i.e. the second term of the integral expression of $\Psi_\beta(r)$ vanishes when $\beta=0$ since the interval of integration degenerates to a point. 

Direct calculation shows  
\begin{eqnarray}
\Psi_\beta''(r)&\circeq&\displaystyle \frac1{\cos\beta}\left[\frac{h}{\sqrt{r^2-h^2}}+\frac{h\sin\beta}{\sqrt{r^2-h^2\sin^2\beta}} \right. \label{last_Psi''_II-1} \\[2mm]
&&\displaystyle \phantom{\frac1{\cos\beta}\left[\right]}
\left. +\left(\pi-\arctan\left(\frac{h\sin\beta+\sqrt{r^2-h^2\sin^2\beta}\,\cot\beta}{\sqrt{r^2-h^2\sin^2\beta}-h\cos\beta}\right)-\arctan\frac{h}{\sqrt{r^2-h^2}}\right)\right]. \label{last_Psi''_II}
\end{eqnarray}
As $r$ decreases to $h$, the last term in the right hand side of $\Psi_\beta''(r)$, \eqref{last_Psi''_II}, tends to $0$. Therefore  
\[
\Psi_\beta''(r) \,\circeq\,  \displaystyle \frac{h}{\cos\beta \sqrt{r^2-h^2}}+\frac{\sin\beta}{\cos^2\beta}+O\left(\sqrt{r-h}\right) \quad (r\searrow h), 
\]
which implies \eqref{Psi''_beta_down_blowup}. 

When $r<h$, $\Psi_\beta(r)=0$, which implies \eqref{Psi''_beta_up}.

\smallskip
(ii) Suppose $\pi/2<\beta\le\pi$. 
Let $t$ denote arclength along $E_j$, measured from $P'$.
Since the behavior when $r$ is close to $h$ is the problem, a range of $t$ close to 0 is sufficient. Let us assume for convenience that it is up to $-h \cos \beta$.

When $r<h$ 
\[
\Psi_{\beta}(r)\,\circeq\,\int_{-(h-r)/\cos\beta}^{-h\cos\beta-\sqrt{r^2-h^2\sin^2\beta}}2\sqrt{r^2-(h+t\cos\beta)^2}\,\,\d t 
+\int_{-h\cos\beta-\sqrt{r^2-h^2\sin^2\beta}}^{-h\cos\beta}\left(\sqrt{r^2-(h+t\cos\beta)^2}+t\sin\beta\right)\d t. 
\]
Then direct calculation shows \eqref{Psi''_beta_up}. 
%\[\lim_{r\nearrow h}\Psi_{\beta}''(r)=O(1).\]

When $r>h$, $\Psi_{\beta}(r)$ is estimated by 
\begin{equation}\label{Psi_{beta}(r)_pi/2<betale_pi_r>h}
\Psi_{\beta}(r)\,\circeq\,\int_{0}^{-h\cos\beta}\left(\sqrt{r^2-(h+t\cos\beta)^2}+t\sin\beta\right)\d t. 
\end{equation}
Direct calculation shows 
\begin{eqnarray}
\displaystyle \Psi_{\beta}''(r)&\,\circeq\,&\displaystyle \frac{-1}{\cos\beta}
\left[-\frac{h}{\sqrt{r^2-h^2}}+\frac{h\sin^2\beta}{\sqrt{r^2-h^2\sin^4\beta}}
+\arctan\left(\frac{h}{\sqrt{r^2-h^2}}\right)-\arctan\left(\frac{h\sin^2\beta}{\sqrt{r^2-h^2\sin^4\beta}}\right)
\right], \qquad{\phantom{a}} \label{Psi''_beta_big_r_big}
\end{eqnarray}
which implies \eqref{Psi''_beta_down_blowup}. 

\medskip
(iii) and (iv) Suppose $\beta<0$. $\Psi_\beta$ is given by

\begin{empheq}[left={\Psi_\beta(r)\,\circeq\,\empheqlbrace}]{align} %\Psi_\beta(r)=
 \displaystyle \int_{0}^{-h\cos\beta+\sqrt{r^2-h^2\sin^2\beta}}\left(\sqrt{r^2-(h+t\cos\beta)^2}+t\sin\beta\right)\d t 
    &\quad \left(-\frac\pi2<\beta<0, \> r>h \right)  \hspace{-0.4cm} \notag\\
  \mathmakebox[\widthof{$\displaystyle \int_{0}^{-h\cos\beta+\sqrt{r^2-h^2\sin^2\beta}}\left(\sqrt{r^2-(h+t\cos\beta)^2}+t\sin\beta\right)\d t$}][l]{\displaystyle \int_{-h\cos\beta-\sqrt{r^2-h^2\sin^2\beta}}^{-h\cos\beta}\left(\sqrt{r^2-(h+t\cos\beta)^2}+t\sin\beta\right)\d t}
    &\quad \left(-\pi<\beta<-\frac\pi2, \> h|\sin\beta|<r<h \right),   \hspace{-0.4cm}  \notag     \\
  \mathmakebox[\widthof{$\displaystyle \int_{0}^{-h\cos\beta+\sqrt{r^2-h^2\sin^2\beta}}\left(\sqrt{r^2-(h+t\cos\beta)^2}+t\sin\beta\right)\d t$}][l]{\displaystyle \int_{0}^{-h\cos\beta}\left(\sqrt{r^2-(h+t\cos\beta)^2}+t\sin\beta\right)\d t}
    &\quad \left(-\pi<\beta<-\frac\pi2, \> h<r \right).  \label{Psi_beta(r)_-pi<beta<-pi/2}
\end{empheq}
The rest can be obtained by direct calculation. 

\smallskip
(2) It follows immediately from (1), since $\cos\gamma=-\cos\beta$. 
The other cases can be proved by direct computation in the same way. 

\smallskip
(3) follows from (1). 
\end{proof}

\begin{corollary}\label{cor_Psi_Ei-1_Ei_Ei+1_pi/2_beta}
Consider three consecutive sides $E_{i-1}, E_i$ and $E_{i+1}$. 
Assume $a_i$ appears only once in the sub-multiset of $\mathcal{C}$ obtained from $E_{i-1}\cup E_i\cup E_{i+1}$. 
Assume $\t_{i-1}=\pi/2$ and $\beta:=\t_i\ne\pm\pi/2$. 
Then the singular part of $\Psi_{\nu;\,E_{i-1}\cup E_i\cup E_{i+1}}''(r)$ as $r$ decreases to $a_i$ is given by 
\begin{eqnarray}
\displaystyle \Psi_{\nu;\,E_{i-1}\cup E_i\cup E_{i+1}}''(r) &=%\stackrel{r\searrow a_i}{\sim}
& \displaystyle  
-\sqrt{2a_i}\cdot{\tan\beta}\cdot\frac1{\sqrt\rho}+O(1) \hspace{0.5cm}\qquad (\rho=r-a_i). 
\label{blow-up_pi/2_beta_nu}
\end{eqnarray}
\end{corollary}

\begin{proof}
It can be proved in the same way as Corollary \ref{cor_jumps_to_angles}. 
\end{proof}

\begin{corollary}\label{cor:two_sides_vs_one_pi/2}
Suppose, in the situation of Lemma \ref{lemma_tangency_vertex}, $P\in E_i, P'\in E_j$. 
Assume $|i-j|\ge2$ and $P'\in E_{j-1}$. 
Let $\beta'$ be the angle of $E_{j-1}$ from the line $PP'$. 
Assume that no other critical configuration occurs at distance $h$. 
Then 
\[
\Psi_{\nu;\,E_i,E_j\cup E_{j-1}}''(r)=\mp\,(\tan\beta-\tan\beta')\sqrt{\frac h2}\cdot\frac1{\sqrt\rho}+O(1) \qquad (\rho=r-h). 
\]
In particular, the coefficient of the singular part of $\Psi_{\nu;\,E_i,E_j\cup E_{j-1}}''$ at $r=h$ does not vanish. 
\end{corollary}

\begin{proof}
The equality follows directly from \eqref{eq:3-27}. %Lemma \ref{lemma_tangency_vertex} (3). 
By our assumption on polygonal domains, $\beta'\ne \beta, \beta\pm\pi$, which implies that $\tan\beta-\tan\beta'\ne0$. 
\end{proof}

\smallskip
Lastly we consider the case when the level set $L_r$ is tangent to the boundary of $R=[0,a_i]\times[0,a_j]$ at a non-vertex point, namely when $r$ is an interior height (Definition \ref{multiset}).

\begin{corollary}\label{lem_interior_distance} 
Suppose $E_i$ and $E_j$ are not parallel and $H_{ji}=\mbox{pr}_i(P_j)$ is in the interior of $E_i$. 
Let $h=P_jH_{ji}$ and $\beta$ be the angle of $\overrightarrow{P_jP_{j-1}}$ from $\overrightarrow{H_{ji}P_j}$ $(-\pi<\beta\le\pi, \, \beta\ne\pm\pi/2)$ (Figure \ref{generic_interior}). 
Assume that no other critical configuration for the pair $(E_i,E_j)$ occurs at distance $h$.
\begin{figure}[htbp]
\begin{center}
\includegraphics[width=.2\linewidth]{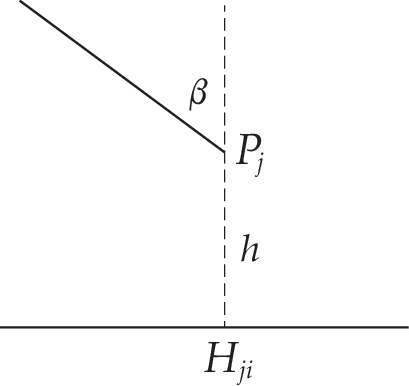}
\caption{Interior height. $\beta\ne\pm\pi/2$}
\label{generic_interior}
\end{center}
\end{figure}
If $\beta\ne0,\pi$, i.e. if $E_j$ is not orthogonal to $E_i$, then $\Psi_{\nu;\,E_i,E_j}''(r)$ blows up to $\pm\infty$ as $r$ decreases to $h$ with the singular part term given by 
\begin{equation}\label{blow-up_interior_height_pi/2_beta}
\Psi_{\nu;\,E_i,E_j}''(r)= 
\mp\sqrt{2h}\cdot{\tan\beta}\cdot\frac1{\sqrt\rho}+O(1) \qquad (\rho=r-h).
\end{equation}
The sign on the right-hand side is given by 
\[
-\mathrm{sgn}\big(E_i,\overrightarrow{H_{ji}P_j}\,\big)\cdot \mathrm{sgn}\big(\overrightarrow{H_{ji}P_j}, E_j\big),
\]
where 
\[
\begin{array}{rcl}
\mathrm{sgn}\big(E_i,\overrightarrow{H_{ji}P_j}\,\big)
&=&\displaystyle \left\{\begin{array}{ll}
+1 & \quad \mbox{if the orientation of $\overrightarrow{H_{ji}P_j}$ is equal to $\pi/2$-rotation of that of $E_i$}, \\
-1 & \quad \mbox{otherwise}
\end{array}
\right. \\[4mm]
\mathrm{sgn}\big(\overrightarrow{H_{ji}P_j}, E_j\big)
&=&\displaystyle \left\{\begin{array}{ll}
+1 & \quad \mbox{if the two orientations of $\overrightarrow{H_{ji}P_j}$ and $E_j$ are compatible at $P_j$}, \\
-1 & \quad \mbox{if both are directed toward $P_j$ or both away from $P_j$}
\end{array}
\right. 
\end{array}
\]
\end{corollary}

\begin{proof} 
It is a consequence of a similar statement for $\Psi_{E_i,E_j}''(r)$, which follows from the same calculation as in Lemma \ref{lemma_tangency_vertex} (1) with $\left(\sqrt{r^2-(h+t\cos\beta)^2}+t\sin\beta\right)$ in the integrand being replaced by  $2\sqrt{r^2-(h+t\cos\beta)^2}$. 
Since the second derivative of the integral of $t\sin\beta$ does not diverge, the singular part of $\Psi_{E_i,E_j}''(r)$ at $r=h$ is twice as large as in Lemma \ref{lemma_tangency_vertex} (1). 
\end{proof}

\begin{corollary}\label{lem_interior_distance_adjacent_sides} 
Suppose neither $E_j$ nor $E_{j+1}$ is parallel to $E_i$ and $H_{ji}=\mbox{pr}_i(P_j)$ is in the interior of $E_i$. 
Let $h=P_jH_{ji}$, and $\beta$ and $\beta'$ be the angles of $P_jP_{j-1}$ and $P_jP_{j+1}$ from $H_{ji}P_j$ respectively. 
Assume that no other critical configuration occurs at distance $h$. 
Then $\Psi_{\nu;\,E_i,E_j\cup E_{j+1}}''(r)$ blows up to $\pm\infty$ as $r\searrow h$ with 
\begin{equation}\label{blow-up_interior_height_beta_beta'}
\Psi_{\nu;\,E_i,E_j\cup E_{j+1}}''(r)= 
\mp\sqrt{2h}\,\big(\tan\beta-\tan\beta'\big)\,\frac1{\sqrt\rho}+O(1) \qquad (\rho=r-h),
\end{equation}
where the sign in the right hand side is the same as in the previous corollary. 
\end{corollary}

\begin{proof}
Note that the signs of the first term of the right hand side of \eqref{blow-up_interior_height_pi/2_beta} for $\Psi_{\nu;\,E_i,E_j}''$ and $\Psi_{\nu;\,E_i,E_{j+1}}''$ are opposite since 
\[
\mathrm{sgn}\big(\overrightarrow{H_{ji}P_j}, E_j\big)
=-\mathrm{sgn}\big(\overrightarrow{H_{ji}P_j}, E_{j+1}\big).
\]
Then \eqref{blow-up_interior_height_pi/2_beta} implies \eqref{blow-up_interior_height_beta_beta'}. 

Since $\t_j\ne0$ and $\Ga$ is simple 
by our assumption, $\beta'\ne \beta, \beta\pm\pi$, which implies that $\tan\beta-\tan\beta'\ne0$. 
\end{proof}

%!!!!!!!!!!!!!!!!!!!!!!!!!!!!!!!!!!!!!%%%%%%%%%%%%%%%%%%%%%%%%%%%%%%%%%%%%
\subsubsection{The degenerate case: parallel sides}\label{subsub:degenerate}
%!!!!!!!!!!!!!!!!!!!!!!!!!!!!!!!!!!!!!%%%%%%%%%%%%%%%%%%%%%%%%%%%%%%%%%%%%
We study the case when the sides are parallel. In this case, unions of parallel lines instead of nondegenerate ellipses appear as the distance level sets. 

Suppose $E_i$ and $E_j$ are parallel. 
Let $h$ be the distance between the lines containing them, 
and $l$ $(l\ge0)$ be the length of $E_i\cap \mbox{\rm pr}_{i}(E_j)$, which is equal to the length of $E_j\cap \mbox{\rm pr}_{j}(E_i)$. 

\begin{figure}[htbp]
\begin{center}
\includegraphics[width=.9\linewidth]{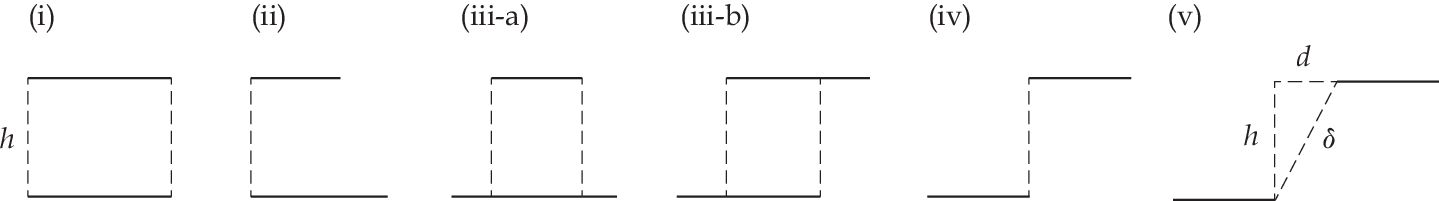}
\caption{$E_i$ and $E_j$. The central two figures correspond to the third case.}
\label{parallel}
\end{center}
\end{figure}

In cases (i)--(iii), i.e. when $l$ is positive, the two endpoints of the overlap give two distinct critical-length instances equal to $h$. 
Hence these cases are excluded if we assume that $\mathcal C$ has no repeated elements. 
Case (v) is already covered by Proposition \ref{lem_jump_Psi_endpoints}. 

In Case (iv), i.e. when $E_i\cap \mbox{\rm pr}_{i}(E_j)$ is a singleton, 
\[
\Psi_{E_i,E_{j}}(r)=\frac12(r^2-h^2) \qquad (r>h \text{ near } h), 
\]
which implies 
\[
J^{(1)}_{\nu;i,j}(h)=\pm h\neq0.
\]

%!!!!!!!!!!!!!!!!!!!!!!!!!!!!!!!!!!!!!%%%%%%%%%%%%%%%%%%%%%%%%%%%%%%%%%%%%
\subsection{Geometric information on polygons obtained from $\nu$-weighted IDD}
%!!!!!!!!!!!!!!!!!!!!!!!!!!!!!!!!!!!!!%%%%%%%%%%%%%%%%%%%%%%%%%%%%%%%%%%%%

We give some of the properties of polygons that can be found from the discontinuity of the $\nu$-weighted interpoint distance distribution $\Psi_{\nu;\,\Ga}''$. 
\begin{proposition}\label{from_Psi_to_polygon}
$\Psi_{\nu;\Ga}$ is smooth on $\RR_{+}\setminus\mathcal{C}$, where $\mathcal{C}$ is the multiset of critical lengths given by \eqref{jump_multiset}. 
If $\mathcal C$ has no repeated elements and at most one exterior angle is equal to $\pm\pi/2$, then $\Psi_{\nu;\Gamma}$ fails to be $C^2$ at $r=c>0$ if and only if $c\in\mathcal C$.
\end{proposition}

It follows that the lengths of sides and diagonals and interior heights can be obtained as points of discontinuity of $\Psi_{\nu;\Ga}''$. 

\begin{proof}
The first assertion comes from the fact that $\Psi_{\nu;\,\Ga}=\sum_{i,j}\Psi_{\nu;\,E_i,E_j}$ and the smoothness of $\Psi_{\nu;\,E_i,E_j}$ may fail only at critical lengths in $\mathcal{C}$ (Proposition \ref{prop_smoothness_outside_C}). 

Let $c$ be an element of the multiset $\mathcal C$. Then it gives a nonzero local contribution to the discontinuity of $\Psi_{\nu;\Ga}''$ at $r=c$ as is stated in Lemma \ref{lemma_tangency_vertex}, Corollaries \ref{cor_jumps_to_angles}, \ref{cor_jump_diagonal}, \ref{cor:two_sides_vs_one_pi/2}, \ref{cor_Psi_Ei-1_Ei_Ei+1_pi/2_beta} and \ref{lem_interior_distance_adjacent_sides} according to the type of the critical length. 
 
If two distinct critical configurations gave the same critical length, their jumps or leading singular terms could cancel; for example, contributions corresponding to angles differing by $\pi$ may have opposite signs. 
However, under the assumption that $\mathcal C$ has no repeated elements, no other critical configuration contributes at the same value $r=c$. Hence the nonzero local singularity associated with $c$ cannot be cancelled. 
Therefore the continuity of $\Psi_{\nu;\Gamma}''$ fails at $r=c$.

In the singleton-overlap case of Subsection \ref{subsub:degenerate}, the two ordered pairs $(E_i,E_j)$ and $(E_j,E_i)$ give a total jump $2\langle\nu_i,\nu_j\rangle h\ne0$ in the first derivative. 
All other contributions associated with the same vertex pair (i.e. the contributions of the other three possible pairs of sides) are $C^1$ at $h$, so this jump cannot cancel. 
\end{proof}

Next we summarize the information on the exterior angles, which can be determined from the discontinuity of $\Psi_{\nu;\,\Ga}''$ at side lengths in case the multiset of critical lengths $\mathcal{C}$ has no duplications. 

\begin{proposition}\label{Psi<->angles}
Assume there is no duplication in the multiset $\mathcal{C}$. 
Assume in addition that at most one exterior angle is equal to $\pm\pi/2$. 
Then the signed exterior angles adjacent to the side of length $a_i$ are given as follows.

\begin{enumerate}
\item Suppose $J^{(2)}_{\nu;\,\Ga}(a_i)=\pm\infty$. Put $\rho=r-a_i$. Then 
exactly one of the two exterior angles is a right angle. 
Suppose the sign of the right angle is positive. 
If 
\[
  \Psi_{\nu;\Gamma}''(a_i+\rho)
  =C_i\rho^{-1/2}+O(1) \quad (\rho\searrow0),
\]
then the other exterior angle $\beta$ satisfies 
\[
  \tan\beta=-\frac{C_i}{\sqrt{2a_i}},
\]
and is therefore determined modulo $\pi$. 
Without fixing the sign of the right angle, there is an additional simultaneous sign ambiguity. 

\item Suppose $J^{(2)}_{\nu;\,\Ga}(a_i)$ is finite. 
Put $S=J^{(2)}_{\nu;\,\Ga}(a_i)$ and $T=J^{(3)}_{\nu;\,\Ga}(a_i)$. 
Then 
\[
\{\tan\t_{i-1}, \tan\t_i\}=\left\{U+V,U-V\right\}\>\mbox{ or }\>\left\{-(U+V), -(U-V)\right\},
\]
where $U$ and $V$ are given by 
\[
U=\frac12\sqrt{-a_i\frac TS-\frac{3S}2-1}, \quad V=\frac12\sqrt{-a_i\frac TS+\frac{S}2-1}\,.
\]
\end{enumerate}
\end{proposition}

Note that, in (2), $J^{(2)}_{\nu;\Gamma}(a_i)$ is finite if and only if neither $\theta_{i-1}$ nor $\theta_i$ equals $\pm\pi/2$ by Corollaries \ref{cor_jumps_to_angles} and \ref{cor_Psi_Ei-1_Ei_Ei+1_pi/2_beta}, and in that case $J^{(3)}_{\nu;\Gamma}(a_i)$ is finite as well by Corollary \ref{cor_jumps_to_angles}.

\begin{proof}
(1) follows from \eqref{blow-up_pi/2_beta_nu}. 

(2) follows from 
\[
\tan\t_{i-1}\tan\t_i=-\frac{S}2, \quad 
(\tan\t_{i-1}+\tan\t_i)^2=-a_i\frac TS-\frac{3S}2-1, 
\]
which is the consequence of \eqref{jump_Psi_E0E1E2_nu''} and \eqref{jump_Psi_E0E1E2_nu'''}. 
\end{proof}

We remark that it is more complicated to extract angle information from the discontinuity in $\Psi_{\Ga}''$ rather than $\Psi_{\nu;\Ga}''$. 

%!!!!!!!!!!!!!!!!!!!!!!!!!!!!!!!!!!!!!%%%%%%%%%%%%%%%%%%%%%%%%%%%%%%%%%%%%
\section{Reconstruction of generic polygonal domains}\label{sect_4}
%!!!!!!!!!!!!!!!!!!!!!!!!!!!!!!!!!!!!!%%%%%%%%%%%%%%%%%%%%%%%%%%%%%%%%%%%%

We give a theorem that includes an analogue of Waksman's theorem. 
Recall $\mathcal{C}$ is a multiset of the critical lengths, i.e. side lengths, diagonal lengths, and interior heights 
(Definition \ref{multiset}). 
Let $\mathcal{A}=[\t_1,\dots,\t_n]$ be the multiset of the signed exterior angles. 
We call $P_jH_{ji}$ $(j=i-2,i+1)$ an {\em exterior height at an adjacent vertexa} if $H_{ji}\not\in E_i$. 

\begin{definition}\label{def_generic} \rm 
A polygonal domain $\dom$ is {\em generic} if 
\begin{enumerate}
\item The equation \> $\displaystyle \sum_{c_i\in\mathcal{C}}\e_i c_i=0$ $\>(\e_{i}=0,\pm1)$ \> does not have a non-trivial solution. 

(This implies that the multiset $\mathcal{C}$ has no repeated elements, namely, it is in fact a set.)
\item No exterior height at an adjacent vertex coincides with a critical length in $\mathcal{C}$. 
\item $|\tan\t_i|$ $(\t_i\in\mathcal{A})$ are all distinct. 
(We agree that $|\tan(\pm\pi/2)|=\infty$.) 
\end{enumerate}
\end{definition}
A multiset that satisfies the condition (1) is said to be a {\em dissociated} set. 

\begin{proposition}
Fix $n\geq 3$, and let $\mathcal P_n$ denote the set of ordered
$n$-tuples
\[
(P_1,\ldots,P_n)\in \left(\mathbb R^2\right)^n
\]
which form a simple closed polygonal curve and for which no two adjacent sides are collinear. 
Then the set of non-generic polygons is contained in a finite union of proper real-analytic subsets of $\mathcal P_n$. 
Hence the set of generic polygons contains an open dense subset of $\mathcal P_n$ of full Lebesgue measure. 
\end{proposition}

\begin{proof}[Sketch of proof]
For each of conditions (1)--(3) in Definition \ref{def_generic}, failure of the condition is described by finitely many nontrivial real-analytic equations in the vertex coordinates. 
Hence the set of non-generic polygons is contained in a finite union of proper real-analytic subsets of $\mathcal P_n$. 
Since a proper real-analytic subset is closed, nowhere dense, and measure zero, the conclusion follows. 
\end{proof}

\begin{theorem}\label{thm_generic}
Let $\Omega_1$ and $\Omega_2$ be generic planar polygonal domains. If their interpoint distance distributions coincide, then $\Omega_1$ and $\Omega_2$ are related by a Euclidean isometry. 
By Theorem \ref{thm_equivalence}, the same conclusion holds if any one of the other equivalent invariants coincides. 
\end{theorem}

The strategy of proof is given in Table \ref{table}. 

\begin{proof}
By Proposition \ref{prop:equiv_via_bar_g}, equality of the body IDDs implies equality of the normal-weighted boundary IDDs. 
We show that this common function determines the polygon up to Euclidean isometry. 

First we determine the side lengths. 
By Proposition \ref{from_Psi_to_polygon} and the conditions (1) and (3) of genericity we obtain the set of critical lengths $\mathcal{C}$. 
By Corollary \ref{perimeter} we obtain the perimeter of $\Omega$. 
Then by condition (1) of genericity the set of side lengths is determined. Let it be $\mathcal{S}=\{l_\la\}_{1\le\la\le n}$, which is the same as $\{a_i\}_{1\le i\le n}$ as a set. 
At this moment, we do not know the cyclic order of side lengths $l_\la$ in the configuration of the polygon. 

Next we determine the set $\{\tan\t_1,\dots,\tan\t_n\}$. 
The condition (3) of genericity implies that $\pm\pi/2$ can appear in $\mathcal{A}$ at most once. 
Whether $\pm\pi/2$ appears in $\mathcal{A}$ can be determined by whether $\Psi_{\nu;\Gamma}''$ diverges at some elements of $\mathcal{S}$. By relabeling the indices if necessary, we may assume these elements to be, for example, $l_1$ and $l_n$. 
In this case, by applying a reflection if necessary, we may assume that the angle is $\pi/2$. 

Define $i(\la)\in\{1,\dots,n\}$ so that $E_{i(\la)}$ is the side with length $l_\la$ $(1\le\la\le n)$. 
If $\Psi_{\nu;\Gamma}''$ does not blow up at $r=l_\la$, i.e. none of the exterior angles $\t_{i(\la)-1}, \t_{i(\la)}$ at the endpoints of the side $E_{i(\la)}$ is equal to $\pm\pi/2$, the pair $\{\tan \t_{i(\la)-1}, \tan\t_{i(\la)}\}$ is determined up to sign by Proposition \ref{Psi<->angles} (2). 
Fix the signs of $\tan\t_j$'s inductively as follows. 

Let us start with $\la=1$. If $\t_{i(n)}=\pi/2$ then the sign of $\tan\t_{i(1)}$ is fixed by Proposition \ref{Psi<->angles} (1), otherwise choose either choice of sign of $\{\tan \t_{i(1)-1}, \tan\t_{i(1)}\}$. 
Note that changing the sign of $\tan\t_{i(1)}$ can be realized by reflection. 
Next, take the unique element $\la'\in\{2,\dots,n\}$ ($\{2,\dots,n-1\}$ if $\t_{i(n)}=\pi/2$) such that one of $|\tan\t_{i(\la')-1}|$ and $|\tan\t_{i(\la')}|$ is equal to $|\tan\t_{i(1)}|$. 
Namely, two sides are adjacent if and only if the pairs of tangents attached to them share a common value. 
The uniqueness of $\la'$ is a consequence of the condition (3) of genericity. 
Fix the sign of $\{\tan\t_{i(\la')-1}, \tan\t_{i(\la')}\}$ so that it is consistent with the sign of $\tan\t_{i(1)}$. 
Thus we can inductively fix all the signs of $\tan\t_j$'s and the cyclic order of $l_\la$'s. 

Finally, we resolve the $\pi$-ambiguity of $\t_j$ as follows. 
If $\theta_i\ne\pi/2$ then $\pi/2<|\t_i|<\pi$ if and only if $\min\{a_i,a_{i+1}\}|\sin\t_i|$ appears in $\mathcal{C}$. 
Indeed, if $\pi/2<|\t_i|<\pi$ and if $\min\{a_i,a_{i+1}\}=a_i$, say, then the foot of the perpendicular from $P_{i-1}$ to $l_{i+1}$ lies in $\textrm{Int}E_{i+1}$, 
whereas if $0<|\t_i|<\pi/2$ then $\min\{a_i,a_{i+1}\}|\sin\t_i|$ is an exterior height at an adjacent vertex and hence is not an element of $\mathcal{C}$. 
We remark that the genericity condition (2) is imposed precisely to rule out the possibility that an angle with $\vert{}\theta_i\vert{} < \pi/2$ is misidentified as $\pi/2 < \vert{}\theta_i\vert{} < \pi$ due to an accidental coincidence between $\min\{a_i,a_{i+1}\}\vert{}\sin\theta_i\vert{}$ and an unrelated critical length in $\mathcal{C}$.

Now the cyclic order of the side lengths and the signed exterior angles are determined, 
the polygonal domain $\Omega$ can be reconstructed up to isometry. 
\end{proof}

\begin{remark}\rm 
The genericity conditions are not fully detectable from $\Psi_\Omega$ alone. 
We do not claim that the genericity conditions of Definition \ref{def_generic} can be detected from $\Psi_\Omega$ itself. By Proposition \ref{prop_smoothness_outside_C} the set
$$\mathcal S=\{\,r>0 : \Psi_{\nu;\Gamma}\ \text{is not of class}\ C^2\ \text{at}\ r\,\}$$
satisfies $\mathcal S\subseteq\mathcal C$ for every polygonal domain, and Proposition \ref{from_Psi_to_polygon} gives $\mathcal S=\mathcal C$ when $\mathcal C$ has no repeated elements. 
When $\mathcal C$ does have repeated elements, however, the contributions of distinct critical configurations at the same value of $r$ may cancel. 
Indeed, by \eqref{formula_J_algebraic} the jumps \eqref{jump_nu} and \eqref{jump_nu_3} attached to $(\alpha,\beta)$ and to $(\alpha,\beta\pm\pi)$ have equal absolute values and opposite signs, and the sum of the two local contributions is in fact smooth at $r=b$, as the argument of Step 2 in the proof of Proposition \ref{lem_jump_Psi_endpoints} shows. 
Consequently a nongeneric domain may have $\mathcal S\subsetneq\mathcal C$, and the cancelled critical lengths leave no trace in the data. 
Note that cancellation can only occur at a repeated element of $\mathcal C$, hence the hypothesis of Proposition \ref{from_Psi_to_polygon} is essetial.

Theorem \ref{thm_generic} is therefore a uniqueness statement within the class of generic polygonal domains. Whether a generic polygonal domain can share its interpoint distance distribution with a nongeneric one is left open here. 

On the other hand, several quantities are determined without genericity assumption. 
First, $\mathrm{Vol}(\Omega)$ and $\operatorname{diam}\Omega$ can be read off from $\Psi_\Omega$. 
Furthermore, by Corollary \ref{Cor_Psi_r_small},  
$\Psi_{\nu;\Gamma}$ is equal to the quadratic polynomial $2\,\mathrm{Length}(\Gamma)\,r+K r^{2}$ on $[0,c_0)$, where $K=\sum_{i}(\theta_i\cot\theta_i-1)$, which implies that the perimeter and $K$ are determined, although the number of vertices is not determined by this expansion alone. 

\end{remark}

Combined with the equivalence given in Theorem \ref{thm_equivalence}, we obtain an analogue of Waksman's theorem. 

\begin{corollary}
Let $\Omega_1$ and $\Omega_2$ be generic convex polygonal domains. 
If their chord length distributions coincide, then they are related by a Euclidean isometry. 
\end{corollary}

Finally, we remark that this study is a continuation of the author's recent research on spatial discrimination. On the one hand, maximally symmetric spaces such as balls can be distinguished by geometric invariants like Riesz energy functionals, thanks to the isoperimetric inequality (\cite{Oball}); on the other hand, sufficiently generic (asymmetric) finite metric spaces are also distinguishable by invariants such as the magnitude (\cite{Omag1}) and the $q$-spectrum (\cite{Omag3}). However, among spaces exhibiting intermediate symmetry, counterexamples exist that cannot be distinguished by such invariants.

Jun O'Hara

Department of Mathematics and Informatics, Faculty of Science, 
Chiba University

1-33 Yayoi-cho, Inage, Chiba, 263-8522, JAPAN.  

E-mail: ohara@math.s.chiba-u.ac.jp


\begin{thebibliography}{OS} 
\bibitem{AB}G.~Averkov and G.~Bianchi, {\em Confirmation of Matheron's conjecture on the covariogram of a planar convex body}, J. Eur. Math. Soc. 11 (2009) 1187\,--\,1202. 

\bibitem{BG}M.~Baake and U.~Grimm, {\em Aperiodic Order}. Volume 1: A Mathematical Invitation, Cambridge University Press, Cambridge (2013).


\bibitem{BBD}C.~Benassi, G.~Bianchi and G.~D'Ercole, {\em Covariogram of non-convex sets}, Mathematika 56 (2010), 267\,--\,284.

\bibitem{Bia02}G.~Bianchi, {\em Determining convex polygons from their covariograms}. 
Adv. Appl. Probab. 34 (2002), 261\,--\,266. 

\bibitem{Bianchi}G.~Bianchi, {\em The covariogram problem}, in Harmonic Analysis and Convexity, Adv. Anal. Geom., 9, A.~Koldobsky and A.~Volberg eds., De Gruyter (2023) 37\,--\,82.

\bibitem{B2} W.~Blaschke, Integralgeometrie 2: Zu Ergebnissen von M.W. Crofton. Bull. Math. Soc. Roum. Sci. 37 (1935), 3\,--\,11. 


\bibitem{B} J.-L.~Brylinski, {\em The beta function of a knot.} Internat. J. Math. {10} (1999), 415\,--\,423.

\bibitem{CB} A.~J.~Cabo and A.J.~Baddeley, {\em Line transects, covariance functions and set convergence}. Adv. Appl. Prob. 27 (1995), 585\,--\,605. 

\bibitem{C09}S.~Ciccariello, {\em The correlation functions of plane polygons}, J. Math. Phys. 50, 103527 (2009).

\bibitem{C10}S.~Ciccariello, {\em The isotropic correlation function of plane figures: the triangle case}, J. Phys. Conf. Ser. 247, 012014 (2010).

\bibitem{C20}S.~Ciccariello, {\em The chord-length distribution of a polyhedron}, 
Acta Crystallogr. A 76 (2020), 474\,--\,488.

\bibitem{F}H.~Federer, {\em Geometric measure theory}, Springer, Berlin, Heidelberg, New York (1969).

\bibitem{FV}E. J.~Fuller and M.K.~Vemuri. {\em The Brylinski Beta Function of a Surface.}  Geometriae Dedicata 179 (2015),  153\,--\,160. 

\bibitem{GGZ}R.~J.~Gardner, P.~Gronchi and C.~Zong, {\em Sums, projections, and sections of lattice sets, and the discrete covariogram}, Discrete Comput. Geom. 34 (2005), 391\,--\,409.

\bibitem{G}J.~Gates, {\em Recognition of triangles and quadrilaterals by chord length distribution}, J. Appl. Prob. 19 (1982), 873\,--\,879.

\bibitem{GS}I.M. Gel'fand and G.E. Shilov, {\em Generalized Functions. Volume I: Properties and Operations}, Academic Press, New York and London (1964). 

\bibitem{GGV}I. M. Gel'fand, M. I. Graev and N. Ya Vilenkin, {\em Generalized Functions, Volume 5, Integral Geometry and Representation Theory}, Academic Press, New York and London (1966).


\bibitem{LXYZ}E.~Lutwak, D.~Xi, D.~Yang, and G.~Zhang, {\em Chord measures in integral geometry and their Minkowski problems}, CPAM 77 (2024), 3277\,--\,3330.

\bibitem{MC}C.~L.~Mallows and J.~M.~C.~Clark, {\em Linear-Intercept Distributions Do Not Characterize Plane Sets.} J. Appl. Prob. 7 (1970), 240\,--\,244.

\bibitem{Maple}Maplesoft, a division of Waterloo Maple Inc., \textit{Maple 2026}. Waterloo, Ontario, 2026. \url{https://www.maplesoft.com}.

\bibitem{M}B.~Mat\'ern, {\em Spatial variation}, Springer-Verlag, Berlin (1985). 

\bibitem{Matheron}G.~Matheron, {\em Random Sets and Integral Geometry}, Wiley, New York (1975).

\bibitem{M2}G.~Matheron, {\em Le covariogramme g\'eom\'etrique des compacts convexes de $\RR^2$}, Technical Report N-2/86/G, Centre de G\'eostatistique, \'Ecole Nationale Sup\'erieure des Mines de Paris (1986). 

\bibitem{Mo}I.~Molchanov, {\em Theory of Random Sets}, 2nd edition, Springer-Verlag, London (2017). 

\bibitem{N}W.~Nagel, {\em Orientation-Dependent Chord Length Distributions Characterize Convex Polygons}, J. Appl. Prob. 30 (1993), 730\,--\,736.

\bibitem{Oball}J.~O'Hara, {\em Characterization of balls by generalized Riesz energy,} Math. Nachr. 292 (2019), 159\,--\,169.

\bibitem{Omag1}J.~O'Hara, {\em Magnitude function identifies generic finite metric spaces}, Discrete Analysis (2025), DOI 10.19086/da.143788.

\bibitem{Omag3}J.~O'Hara, {\em Distinguishing finite metric spaces via similarity spectra}, arXiv:2502.08980.

\bibitem{OS2}J.~O'Hara and G.~Solanes, {\em Regularized Riesz energies of submanifolds,} Math. Nachr. 291 (2018), 1356\,--\,1373.

\bibitem{P}B.~Petkantschin, {\em Integralgeometrie 6. Zusammenh\"ange zwischen 
den Dichten der linearen Unterrume im n-dimensionalen Raum.} Abh. Math. Semin. Univ. Hambg. 11, Issue 1, (1935), 249\,--\,310.

\bibitem{San2}L.A.~Santal\'o, {\em Integral Geometry and Geometric Probability}, Addison-Wesley Publishing Company, 1976.

\bibitem{Sch}M.~Schmitt, {\em On two inverse problems in mathematical morphology}, in Mathematical Morphology in Image Processing, ed. E. R. Dougherty, Marcel Dekker (1993), 151\,--\,169.

\bibitem{SW}R.~Schneider and W.~Weil, {\em Stochastic and Integral Geometry}, Springer, Berlin (2008).

\bibitem{W}P.~Waksman, {\em Polygons and a conjecture of Blaschke's}, Adv. Appl. Prob. 17 (1985), 774\,--\,793.


\end{thebibliography}
\end{document}